\documentclass[11pt]{article}
\usepackage{amsmath,amssymb,amsbsy,amsthm,mathtools}
\usepackage[usenames,dvipsnames]{xcolor}
\usepackage{times} 

\usepackage{braket}
\usepackage{cmap} 
\usepackage{mparhack} 
\usepackage{enumerate}

\usepackage{bm}

\usepackage[expansion=true, protrusion=true]{microtype} 

\usepackage{subcaption}
\usepackage{todonotes}

\usepackage{tikz}

\newcommand{\vecb}[1]{\boldsymbol{#1}}

\newcommand{\Rset}{\mathbb{R}} 

\newcommand{\bigbraces}[1]{\big\{#1\big\}}

\newcommand{\absval}[1]{\lvert #1\rvert}
\DeclareMathOperator{\argmin}{argmin}

\newcommand{\discr}[1]{{#1}_n} 
\newcommand{\Xspace}{\mathbb{X}}  
\newcommand{\dXspace}{\discr{\Xspace}}
\newcommand{\Yspace}{\mathbb{Y}}  
\newcommand{\dYspace}{\discr{\Yspace}}
\newcommand{\Mcurves}{\mathcal{M}}  
\newcommand{\dMcurves}{\discr{\Mcurves}}
\newcommand{\curve}{m}            
\newcommand{\dcurve}{\vecb{\curve}}
\newcommand{\dConstr}{\varPhi}
\newcommand{\Sb}{\mathbb{S}^1}             
\newcommand{\Angle}{\theta}
\newcommand{\dAngle}{\vecb{\Angle}}
\newcommand{\AngleSpace}{\varTheta}
\newcommand{\dAngleSpace}{\discr{\AngleSpace}}
\newcommand{\BasePoint}{p}
\newcommand{\Interval}{[0,1]}
\newcommand{\energy}{\mathcal{E}_{\mathrm{b}}} 
\newcommand{\dEnergy}{\mathcal{E}_{\mathrm{b,n}}}
\newcommand{\penetrEnerg}{\mathcal{E}_{\mathrm{M}}} 
\newcommand{\dPenetrEnerg}{\mathcal{E}_{\mathrm{M,n}}}
\newcommand{\totEnerg}{\mathcal{E}} 
\newcommand{\dTotEnerg}{\discr{\totEnerg}}
\newcommand{\du}{\vecb{u}}
\newcommand{\dv}{\vecb{v}}
\newcommand{\dx}{\vecb{x}}
\newcommand{\dy}{\vecb{y}}
\newcommand{\xdag}{\curve^{\dagger}} 
\newcommand{\data}{y}             
\newcommand{\uinf}{u_{\infty}}    
\newcommand{\McurvesPhase}{\Mcurves_{\mathrm{ampl}}} 
\newcommand{\Fphase}{F_{\mathrm{ampl}}} 
\newcommand{\Uinf}{U_{\infty}}

\newtheorem{theorem}{Theorem}
\newtheorem{proposition}[theorem]{Proposition}
\newtheorem{lemma}[theorem]{Lemma}
\newtheorem{remark}[theorem]{Remark}

\usepackage[final,%
	pdftex,%
	breaklinks=true,%
]{hyperref}

\newcommand{\qand}{\quad \text{and} \quad}

\newcommand{\cE}{{\mathcal{E}}}

\newcommand{\cJ}{{\mathcal{J}}}

\DeclareMathOperator{\Hess}{Hess}
\newcommand{\dd}{{\mathrm{d}}}

\newcommand{\ii}{{\mathrm{i}}}

\newcommand{\ceq}{\coloneqq}

\newcommand{\R}{{\mathbb{R}}}
\newcommand{\C}{{\mathbb{C}}}
\newcommand{\N}{\mathbb{N}}
\newcommand{\Z}{{\mathbb{Z}}}

\newcommand{\nabs}[1]{\lvert{#1}\rvert} 

\newcommand{\norm}[1]{\left\lVert#1\right\rVert}
\newcommand{\nnorm}[1]{\lVert{#1}\rVert}
\newcommand{\bignorm}[1]{\big\lVert{#1}\big\rVert}

\newcommand{\ninnerprod}[1]{\langle{#1}\rangle}

\newcommand{\paren}[1]{\left(#1\right)}

\newcommand{\bigparen}[1]{\big(#1\big)}
\newcommand{\biggparen}[1]{\bigg(#1\bigg)}

\newcommand{\bracket}[1]{\left[#1\right]}

\newcommand{\intervalcc}[1]{\left[#1\right]}

\DeclareMathOperator{\grad}{grad}  

\begin{document}
\title{Elastic energy regularization for inverse obstacle scattering problems}

\author{J. Eckhardt\footnote{Institute for Numerical and Applied Mathematics, University of G\"ottingen, Germany}\ , R. Hiptmair\footnote{Seminar for Applied Mathematics, ETH Zurich, Switzerland}\ , T. Hohage$^*$, H. Schumacher\footnote{Institute for Mathematics, RWTH Aachen University, Germany}\ , M. Wardetzky$^*$}

\maketitle

\begin{abstract}
By introducing a shape manifold as a solution set to solve inverse obstacle scattering problems we allow the reconstruction of general, not necessarily star-shaped curves. 
The bending energy is used as a stabilizing term in Tikhonov regularization to gain independence of the parametrization. Moreover, we discuss how self-intersections can 
be avoided by penalization with  the M\"obius energy and prove the regularizing property of 
our approach as well as convergence rates under variational source conditions.

In the second part of the paper the discrete setting is introduced, and we describe a numerical method for finding the minimizer of the Tikhonov functional on a shape-manifold. Numerical examples demonstrate the feasibility of reconstructing non-star-shaped obstacles.

\bigskip
\noindent \textbf{Keywords:} Shape spaces, inverse obstacle scattering, nonlinear 
Tikhonov regularization
\end{abstract}


\section{Introduction}

Inverse obstacle scattering problems consist of reconstructing the shape of an 
impenetrable or homogeneous scattering obstacle from measurements of scattered waves. 
Such problems, which occur for example in structural health monitoring and medical 
imaging, have been studied intensively, see the monographs 
\cite{CC:06,CK:12,KG:08,potthast:01} and references therein. 

One may distinguish two main classes of reconstruction methods for inverse 
obstacle scattering problems: \emph{Sampling methods} and 
\emph{parameterization-based methods}. In sampling methods, an indicator function 
$f \colon \Rset^d\to \Rset \cup \{\infty\}$ is constructed based on the measured data, 
such that the value of $f$ indicates whether 
a point belongs to the obstacle or not (at least for noise-free data). Examples include the linear sampling method 
\cite{AL:09,CC:06,CK:96}, the factorization method \cite{kirsch:98,KG:08}, 
and the singular source method \cite{potthast:96,potthast:01}. 

In contrast, parameterization-based methods yield a parameterization of 
some approximation of the true obstacle. Examples include 
decomposition methods \cite{CM:86,KK:87} and 
iterative regularization methods \cite{Hohage:99}, in particular, regularized Newton methods \cite{HH:07,MTW:88}  
and nonlinear Tikhonov regularization. 

Both sampling and parameterization-based methods have their respective advantages and disadvantages. 
On the one hand, sampling methods 
do not require a-priori knowledge of the obstacle's topology and often 
not even of the boundary condition. They are typically easy to implement. 
On the other hand, they often require (i) a lot of data (e.g., the complex-valued 
far-field patterns for \emph{all} incident fields); are (ii) less flexible concerning 
modifications of the forward problem (e.g., amplitude measurements or nonlinearities);
and yield (iii) less accurate reconstructions than parameterization-based methods. 
Ideally, both types of methods can complement each other by using a sampling 
method to construct an initial guess for a parameterization-based method 
(see Remark \ref{rem:self_intersection} below).  

For parameterization-based methods, one seeks approximate parameterizations of the unknown curve within a chosen class of boundary curves. In existing literature, this class is often chosen in a rather ad hoc manner. E.g., the obstacle is assumed to be star-shaped 
with respect to some known point such that the boundary can be described by a positive, 
periodic radial function. In this manner, one can formulate inverse obstacle problems as operator equations in Hilbert spaces. In this case, the attendant 
penalty terms are usually represented in the form of Sobolev norms of the parameterization. However, such norms  crucially depend on the choice of the parameterization and thus disobey the geometry of the shape to be reconstructed. Indeed, a single curve $\Sb\to \Rset^2$ admits a continuum of possible parameterizations and therefore, parameterization-dependent norms break symmetry in an unnatural manner.
Moreover, the assumption of star-shaped obstacles is severely restrictive. 

Our contribution is the introduction of the boundary curve's bending energy as a regularizing term for the two-dimensional case. This approach is purely geometric and \emph{independent of the choice of any parameterization} and thus allows for arbitrary planar curves (of sufficient regularity). Considering the set of curves as geometric objects, independent of any particular parameterization, is of course a well established paradigm  by now in the context of shape spaces. Michor and Mumford \cite{MiMu04} showed that the space of closed regular curves (of sufficient regularity) carries the structure of a Riemannian manifold. This ansatz leads to certain degeneracies of geodesics and has therefore subsequently been refined and extended is several directions, e.g., using curvature-weighted $L^2$-metrics \cite{MiMu07}, $L^2$-metrics that incorporate a curve's stretching and bending contributions \cite{SrJaJo06, SrKlJo11}, and certain Sobolev-type metrics \cite{BaBrMa14, BaBrMi15, Br15, BrMiMu14}. In particular, curvature-based (i.e., second order) formulations penalize a curve's bending contributions and lead to physically plausible simulations of thin elastic rods and threads \cite{ACBBGM13, BWRAG08, RuWi12b}. In the discrete setting, curvature-based energies can be approximated using polygonal (piecewise straight) curves. Convergence in Hausdorff distance of the resulting minimizers (under suitable boundary conditions and a length constraint)  to their smooth counterparts was recently shown in \cite{SchuWa19}, which provides one of the theoretical underpinnings of the current work. We discuss the attendant discrete model in Section \ref{sec:discrete}.

We  argue that the use of shape manifolds may be largely preferable over para\-me\-te\-rization-dependent
methods, and the use of bending energy may be beneficial whenever regularization 
is applied to the set of curves in form of a penalization. In this paper 
we focus on Tikhonov regularization since it has the simplest and most complete 
convergence analysis. Our main theoretical result is the regularizing property 
of Tikhonov regularization on shape manifolds and convergence rates under 
variational source conditions. 

The plan of the rest of this paper is as follows: We introduce our
shape-manifold of curves as well as the requisite 
bending energy functional on this manifold in the next section.  
Our main theoretical results on the regularizing 
property and convergence rates of the proposed method are contained in Section \ref{sec:tikhonov}. 
We then introduce the sound-soft obstacle scattering problem 
as a typical example of a forward problem and derive some properties of the 
forward operator defined on the shape manifold in Section \ref{sec:scat}. 
In Section \ref{sec:discrete} we describe our discrete model of the 
shape manifold 
and explain how to solve the associated minimization problem. We finally present our numerical results in Section \ref{sec:results} and combine it with a sampling method in section \ref{sec:sampling}. 


\section{Shape manifold and elastic energy}\label{sec:elastic_energy}

In this section we introduce the shape manifold of closed curves $\varGamma$ in $\R^2$
and investigate its structure. We further define an energy functional on this manifold.

\paragraph*{The shape manifold} 
Let $\varGamma \subset \R^2$ be a regular, closed curve of class $H^2$ of length $L$, i.e., there 
is a parameterization $\gamma \in H^2(\Interval,\R^2)$ satisfying $\gamma'(t) \neq 0$ for all $t \in \Interval$  and the closing conditions
\begin{align}
	\gamma(0) = \gamma(1) \qand \gamma'(0) = \gamma'(1) .
	\label{eq:ClosingConditions}
\end{align}
Without loss of generality, we may assume that $\gamma$ is of constant speed, i.e., $\nabs{\gamma'(t)} = L$.
Thus, we may represent $\gamma$ by a triple $\curve = (\Angle, L, \BasePoint)$ with a base point $\BasePoint \ceq \gamma(0)$, the curve's length $L$, and an angle function $\Angle \in H^1(\Interval, \R)$ via
\begin{align}
	\gamma(t)
	= \gamma_{\curve}(t)
	\ceq \BasePoint + \int_0^t \gamma'(\tau) \, \dd \tau
	= \BasePoint + L \int_0^t \, \bigparen{\cos(\Angle(\tau)),\sin(\Angle(\tau))} \, \dd \tau.
	\label{eq:CurveRepresentation}
\end{align}
In order to fulfill the closing conditions \eqref{eq:ClosingConditions}, $\Angle$ needs to satisfy
\begin{align}
	\int_0^1 \cos(\Angle(t)) \, \dd t = 0,
	\quad
	\int_0^1 \sin(\Angle(t)) \, \dd t = 0	
	\qand
	\Angle(1) - \Angle(0) \in 2 \, \pi \, \Z.
	\label{eq:ClosingConditions2}
\end{align}
The number $\frac{\Angle(1) - \Angle(0)}{2 \, \pi}$ is called the \emph{turning number} of $\gamma$ (not to be confused with the winding number).
A necessary (but not sufficient) condition for $\varGamma$ to be embedded is that $\gamma$
has turning number~$\pm 1$.
Since our application focuses on boundary curves of simply connected domains, we may restrict ourselves to curves of turning number~$+1$ and define the space:
\begin{align*}
	\AngleSpace \ceq \bigbraces{ 
		\Angle \in H^1(\Interval,\R) 
		\,\big|\, 
		\textstyle \int_0^1 \bigparen{\cos(\Angle(t)),\sin(\Angle(t))} \, \dd t = 0	,
		\;
		\Angle(1) - \Angle(0) = 2 \, \pi
	}.
\end{align*}
\begin{lemma}
The space $\AngleSpace$ is an embedded submanifold of $H^1(\Interval,\R)$.
\end{lemma}
\begin{proof}
We may write $\AngleSpace = \set{ \Angle \in E | \varPhi(u) = 0}$ with the affine subspace $E = \set{\Angle \in H^1(\Interval,\R)  | \Angle(1) - \Angle(0) = 2 \, \pi}$
and the smooth mapping $\varPhi \colon E \to \R^2$, 
\begin{align*}
	\varPhi(\Angle) = \int_0^1 \bigparen{\!\cos (\theta(t)), \sin (\theta(t))} \, \dd t.
\end{align*}

By virtue of the implicit function theorem, all that we have to do is to show that $\varPhi$ is a \emph{submersion}, i.e. that $D\varPhi(\Angle)$ admits a bounded linear right inverse for each $\Angle \in E$ (see \cite[Chapter II, \S 2]{Lang:99}).
Notice that $E$ is an affine subspace over the linear subspace
\begin{align*}
	H^1_{\mathrm{per}}(\Interval,\R) 
	\ceq 
	\bigbraces{u\in H^1(\Interval,\R) \, \big|\, u(0)=u(1)}
\end{align*}
and that
\begin{align*}
	D\varPhi \, u
	= \begin{pmatrix}
		-\int_0^1 \sin (\theta(t)) \, u(t) \, \dd t \\
		\hphantom{-} \int_0^1 \cos (\theta(t)) \, u(t) \, \dd t
	\end{pmatrix}.
\end{align*}
Thus, it suffices to construct a $u \in H^1_{\mathrm{per}}(\Interval,\R)$ solving $D\varPhi \, u = \lambda$
and depending linearly on the right hand side for any prescribed $\lambda \in \R^2$.
We set $s(t)\ceq \sin(\theta(t))$ and $c(t)\ceq \cos(\theta(t))$ and make the ansatz $u(t) = a \, s(t) + b \, c(t)$, which leads to the linear equation
\begin{align*}
\begin{pmatrix}
- \langle s,s\rangle &- \langle s,c\rangle \\
 \langle c,s\rangle  &  \langle c,c\rangle 
\end{pmatrix}
\begin{pmatrix}
a \\ b
\end{pmatrix}
= 
\begin{pmatrix}
 \lambda_1 \\  \lambda_2 
\end{pmatrix},
\end{align*}
where  $\langle \cdot ,\cdot\rangle$ denotes the $L^2$ inner product. By Cauchy-Schwarz, the determinant of this system is negative since $\theta(t)$ is continuous and not constant. 
\end{proof}
The tangent space of $\AngleSpace$ is given by
\begin{align*}
	T_\Angle \AngleSpace
	=
	\bigbraces{
		u \in  H^1_{\mathrm{per}}(\Interval,\R) 
		\,\big|\,
		\textstyle
		D \varPhi(\Angle) \, u = 0
	}.
\end{align*}
The family of inner products $(g_\Angle)_{\Angle \in \AngleSpace}$ defined by
\begin{align*}
	g_\Angle(u,v) \ceq \int_0^1 \bigparen{ u(t) \, v(t) + u'(t) \, v'(t) }\, \dd t
	\quad
	\text{for $u$, $v \in T_\Angle \AngleSpace$}
\end{align*}
turns $(\AngleSpace,g)$ into a infinite-dimensional Riemannian manifold (in the sense of~\cite{Lang:99}).

For a compact, convex set of base points $B \subset \R^2$ and for bounds of acceptable curve lengths, we define our space of feasible curves by
\begin{align*}
	\Mcurves \ceq \AngleSpace \times [L_1,L_2] \times B.
\end{align*}
Then $\Mcurves$ is a smooth submanifold with corners in the Hilbert space
\begin{align*}
	\Xspace \ceq H^1(\Interval,\R) \times \R \times \R^2
\end{align*}
and its tangent space at an interior point $\curve = (\Angle, L,\BasePoint)$ is given by
\begin{align*}
	T_m \Mcurves = T_\Angle \AngleSpace \oplus \R \oplus \R^2.
\end{align*}

\paragraph*{Elastic energy} 
Continuing our geometric approach, recall that the Euler-Bernoulli bending energy (see \cite{EulerElasticae}) of a planar curve $\varGamma$ is given by 
\begin{equation*}
 \int_\varGamma \kappa^2\, \dd s,
\end{equation*}
where $\dd s$ is the line element and $\kappa$ denotes the (signed) curvature of $\varGamma$. 
The bending energy, or more precisely the curvature, is a geometrical invariant of the curve 
$\varGamma$ and thus we gain \emph{independence under reparameterizations}, which is the main benefit of our approach. The bending energy models the stored deformation energy of $\varGamma$ under the assumption of an undeformed \emph{straight} rest state of the same length as $\varGamma$.

Let $\varGamma$ be parameterized by $\gamma$ that is represented by $m = ( \Angle, L, \BasePoint) \in \Mcurves$ as in~\eqref{eq:CurveRepresentation}. Then we have $\kappa(t) = \frac{\Angle'(t)}{L}$ and $\dd s = L \, \dd t$.
This shows that bending energy scales with $1/\lambda$ when $\varGamma$ is re-scaled by a factor $\lambda >0$. Thus, without any additional constraints, minimizers of this energy do not exist (the energy of $\gamma_m$ converges to $0$ for $L \to \infty$). We therefore consider the following scale-invariant version $\energy\colon \Mcurves\to [0,\infty)$ of bending energy which is simply the $H^1$-seminorm:
\begin{equation}\label{eq:defi_energy}
\energy(\curve) \ceq \int_0^1 {\Angle'(t)}^2 \, \dd t.
\end{equation}

As mentioned above, $\energy(\curve)$ describes the energy required to deform a \emph{straight} elastic rod 
of length $L$ into $\varGamma$. More generally, consider an undeformed rest state $\varGamma_*$ of non-vanishing curvature (i.e., if $\varGamma_*$ is pre-curved). Assuming that $\varGamma_*$ is deformed into $\varGamma$ by a diffeomorphism $\varphi \colon \varGamma_* \to \varGamma$ 
that does not change the line element\footnote{Notice that for any two (sufficiently regular) planar curves of the \emph{same} total length $L$, there exists a diffeomorphism between them that preserves infinitesimal length at every point. In particular, such a mapping is not necessarily a Euclidean motion.}, bending energy is given by
\begin{equation*}
 \int_{\varGamma_*} (\kappa_*(s) - \kappa(\varphi(s)))^2\, \dd s .
\end{equation*}
Representing $\varGamma_*$ by 
$\curve_*=(\Angle_*,L,\BasePoint_*)$ as above, the scale-invariant version of this energy is given by 
\[
\energy(\curve, \curve_*) = \int_0^1(\Angle'(t)-\Angle_*'(t))^2\, \dd t.
\]
This formulation is useful when $\varGamma_*$ represents a reasonable initial guess that is further optimized in order to obtain the desired solution.

\paragraph*{Non-self-intersecting curves} 
When reconstructing a domain, one requires a boundary curve that is free of self-intersections. In this context, the following lemma is useful.
\begin{lemma}\label{lem:intersection}
The set of non-self-intersecting curves is open in the $\Xspace$-topology. 
\end{lemma}
\begin{proof}
First notice that curves of finite bending energy correspond to elements of the Sobolev space $H^2(\Interval,\R^2)$. Furthermore, by construction, each member of $\Mcurves$ represents a $C^1$-immersion $\gamma \colon \Sb \to \Rset^2$; indeed, due to periodic boundary conditions we can take $\Sb$ as the domain for $\gamma$. Since injective immersions of compact domains are embeddings, we may employ Theorem~3.10 from \cite{MR0198479}, stating that the set of $C^1$-embeddings is open in $C^1(\Sb,\Rset^2)$. Now, the fact that $H^2(\Sb,\Rset^2)$ embeds continuously into $C^1(\Sb,\Rset^2)$ implies the result.
\end{proof}

\begin{remark}\label{rem:self_intersection} 
If a sufficiently good initial guess $\curve_*\in \Mcurves$ 
of the true solution is available and if $\curve_*$ is free of self-intersections, then Lemma \ref{lem:intersection} ensures 
that we can choose 
\begin{equation}\label{eq:M0ball}
\Mcurves_0 \ceq \set{ \curve\in\Mcurves | \|\curve-\curve_*\|_{\Xspace}\leq \delta }
\end{equation}
containing only non-self-intersecting curves. 
In the context of inverse obstacle scattering problems such an initial guess 
can often be constructed by sampling methods as discussed in the introduction and in Section~\ref{sec:sampling}. 
\end{remark}

Although we have not encountered the problem of self-intersections in practice for our method, we briefly outline how to avoid this issue whenever needed. A~popular and widely studied energy that is \emph{self-avoiding} (i.e., finite energy guarantees that the curve is free of self-intersections) is the so-called \emph{M\"{o}bius energy} defined as 
\begin{align}\label{eq:MobiusEnergy}
\penetrEnerg (\varGamma)\ceq
	\int_\varGamma \int_\varGamma
	\biggparen{
	\frac{1}{|x-y|^2 } - \frac{1}{d_\varGamma (x,y)^2}
	} \, \dd s(x) \, \dd s(y) ,
\end{align}
where $d_\varGamma (x,y)$ denotes the geodesic distance between $x$ and $y$ along $\varGamma$ and integration is performed with respect to the line elements.
This parameterization-invariant energy was introduced by  O'Hara \cite{MR1098918} and its analytical properties have been studied by several authors \cite{MR2887901,MR3461038,MR1259363,MR1733697,MR1470748,MR1702037}. 
The self-avoiding property is ensured by the first summand of the integrand, while the second summand is introduced in order to remove the singularity along the diagonal $x=y$. 
The M\"{o}bius energy is invariant under M\"{o}bius transformations (i.e., under conformal transformations of $\C\cong \R^2$) and thus in particular scale-invariant. 
We will show in Section~\ref{sec:tikhonov} that 
using the M\"{o}bius energy as an additional penalty term ensures that minimizers of the regularized problem are indeed free of self-intersections.

\paragraph*{Properties of the energy functionals} 
The analysis of well-posedness and convergence properties of Tikhonov regularization 
in Section \ref{sec:tikhonov} requires  
some properties of the energy functionals $\energy$ and 
$\penetrEnerg$ on the Riemannian manifold $\Mcurves$.
For showing existence of solutions via the direct method of the calculus of variations, 
weakly sequential lower semi-continuity of the objective functional is a desirable property.
Weak convergence, however, is a concept that is \emph{not} invariant under nonlinear changes of coordinates. 
Because we parameterized $\Mcurves$ as in \eqref{eq:CurveRepresentation}, the bending energy becomes a \emph{convex quadratic functional}, enabling us to derive the following result.

\begin{proposition}\label{prop:energy}
Let $\cE \in \set{\energy, \energy(\cdot, m_*)}$. With respect to the $\Xspace$-topology, we have:
\begin{enumerate}[(i)]
\item \label{it:energy:Mclosed} $\Mcurves\subset\Xspace$ is weakly sequentially closed.
\item\label{it:energy:Elsc}  $\cE$ is weakly sequentially lower semi-continuous. 
\item\label{it:energy:Elscompact} 
Modulo shifts by elements of $2 \,\pi \, \Z$ , the sublevel sets $\cE^{-1}([0,a])\subset \Mcurves$ are weakly sequentially compact. 
\end{enumerate}
\end{proposition}

\begin{proof}
We proceed in the usual manner of the direct method of calculus of variations.
In order to show (i), consider a sequence $(\curve_n = (\Angle_n,L_n,\BasePoint_n))_{n \in \N}$ in $\Mcurves$ that converges weakly 
to some $\curve=(\Angle,L,\BasePoint)\in\Mcurves$. 
By the Rellich compactness theorem, $H^1(\Interval,\R)$ is 
compactly embedded in $C(\Interval,\R)$ equipped with the supremum norm.
Thus, weak convergence of $\Angle_n \rightharpoonup \Angle$ in $H^1(\Interval,\R)$ implies strong convergence in $C(\Interval,\R)$.
Since the closing conditions \eqref{eq:ClosingConditions2} are continuous on $C(\Interval,\R)$, this implies that $\Angle \in \AngleSpace$ and thus $\curve \in \Mcurves$.
\\
In order to show (ii), notice that  $\cE$ is defined in terms of a squared seminorm on~$\Xspace$, which is a continuous and convex functional, whose sublevel sets are therefore sequentially closed and convex. The fact that sequentially closed convex sets are 
weakly sequentially closed implies (ii).
\\
For showing (iii), we first observe that for each $z \in 2 \,\pi \, \Z$, the curve represented by $(\Angle + z,L,\BasePoint)$
is the same as the one represented by $(\Angle,L,\BasePoint)$.  
Now let $m_n = (\Angle_n,L_n,\BasePoint_n)$ be a sequence in a sublevel set $\cE^{-1}([0,a])$.
Modulo shifting by $z_n \in 2\, \pi \, \Z$, we may assume that $\Angle_n(0) \in \intervalcc{0,2\pi}$.
We may define an equivalent norm on $H^1(\Interval,\R)$ by
$\|\Angle\|_* \ceq \nabs{\Angle(0)} + \nnorm{\Angle'}_{L^2}$.
We then either have
$\nnorm{\Angle_n}_{H^1} \leq 2 \, \pi + \sqrt{\cE(\Angle_n)}$ (for the case of $\cE = \energy$)
or 
$\nnorm{\Angle_n - \Angle_*}_{H^1} \leq 2 \, \pi + |\Angle_*(0)| + \sqrt{\cE(\Angle_n)}$ (for the case of $\cE = \energy(\cdot, m_*)$).
In either case, the sequence $(\curve_n)_{n \in \N}$ is bounded in $H^1$, and hence it has a subsequence 
$(\Angle_{n_k})$ converging weakly to some $\Angle \in H^1(\Interval,\R)$. 
Moreover, $\intervalcc{L_1,L_2} \times B$ is compact so that we may find a further subsequence so that $m_{n_k}$ converges weakly to some $m = (\Angle, L ,\BasePoint) \in H^1(\Interval,\R) \times \intervalcc{L_1,L_2} \times B$.
Because of (ii), we have $\cE(m) \leq a$ and because of (i), $m$ is indeed an element of $\Mcurves$.
\end{proof}

\begin{lemma}\label{lem:lsc_penetrEnerg}
The M\"{o}bius energy $\penetrEnerg \colon \Mcurves\to [0,\infty]$ defined by \eqref{eq:MobiusEnergy}
is weakly sequentially lower semi-continuous with respect to the weak topology of $\Xspace$.  
\end{lemma}

\begin{proof}
Recall that \eqref{eq:CurveRepresentation} constitutes a smooth mapping from $\Mcurves$ to $H^2(S^1,\R^2)$.
As shown in \cite{MR3461038}, the M\"{o}bius energy is continuously differentiable (and thus continuous) 
on the space of embeddings of class $C^{0,1}(S^1,\R^2) \cap H^{3/2}(S^1,\R^2)$. 
Now the statement follows from the compactness of the embedding of $H^2(S^1,\R^2)$ into this space.
More precisely, let $m_n$, $m \in \Mcurves$ with $m_n \rightharpoonup m$.
We have to show that $\penetrEnerg(m) \leq c \ceq \liminf_{n \to \infty} \penetrEnerg(m_n)$.
In the case of $c = \infty$, there is nothing to show, so assume that
$c < \infty$.
Since~$\penetrEnerg$ is invariant under scaling and translation, we may assume that $L_n = L = 1$ and $\BasePoint_n = \BasePoint = 0$.
Denote by $\gamma_n$, $\gamma \in H^2(S^1,\R^2)$ the corresponding parameterizations.
Due to the Rellich embedding,
we may pick a subsequence such that $c= \lim_{k \to \infty} \penetrEnerg(m_{n_k})$ and such that $\gamma_{n_k} \to \gamma$ strongly in $C^{0,1}\cap H^{3/2}$.
The latter shows that 
\begin{align*}
	\penetrEnerg(\curve) 
	= \lim_{k \to \infty} \penetrEnerg(\curve_{n_k}) = c,
\end{align*}
which proves the claim.
\end{proof}


\section{Tikhonov regularization}\label{sec:tikhonov}
In this section we consider a general injective operator 
\[
F \colon \Mcurves_0\subset\Mcurves\to \Yspace
\]
mapping a set of embedded curves $\Mcurves_0$
into a Hilbert space $\Yspace$. 
The unknown exact solution will be denoted by 
$\xdag\in\Mcurves_0$. 
Noisy data is described by a vector $\data^{\delta}\in\Yspace$  satisfying 
\[
\|\data^{\dagger}-F(\xdag)\|_{\Yspace}\leq \delta.
\] 

In order to approximately recover $\xdag$ from the data $\data^{\delta}$,
we use Tikhonov regularization with some regularization parameter
$\alpha>0$:
\begin{equation}\label{eq:defi_tikh}
\curve_{\alpha}^{\delta}\in \argmin_{\curve\in\Mcurves_0}
\bracket{\|F(\curve)-\data^{\delta}\|_{\Yspace}^2 + \alpha \, \energy(\curve,\curve_*)}.
\end{equation}
Here $\curve_*$ denotes an initial guess of $\xdag$ and may be set to $0$ 
if no such initial guess is available. 
If no appropriate submanifold of embedded curves $\Mcurves_0$ containing 
the true solution is known, we may choose $\Mcurves_0$ by 
setting $\Mcurves_0\ceq \set{\curve\in\Mcurves | \penetrEnerg(\curve) \leq c}$ for sufficiently large $c >0$ 
or alternatively consider Tikhonov regularization of the form 
\begin{equation}\label{eq:defi_tikh_with_penetr}
\curve_{\alpha}^{\delta}\in \argmin_{\curve\in\Mcurves}
\bracket{\|F(\curve)-\data^{\delta}\|_{\Yspace}^2 + \alpha \, \energy(\curve,\curve_*) 
+\alpha \, \penetrEnerg(\curve)}.
\end{equation}
Since $\penetrEnerg(\curve)= \infty$ if $\curve$ is 
self-intersecting,
$\penetrEnerg$ acts as a barrier function: Only the values of $F$ on the set of embedded curves are relevant
and each curve $\gamma_{\curve_{\alpha}^{\delta}}$ is guaranteed to be embedded. 

With the properties of the energy functionals established in the previous section, 
the following convergence properties follow from 
the general theory of nonlinear Tikhonov regularization. 

\begin{theorem}\label{theo:main}
Assume that $\Mcurves_0\subset\Mcurves$ contains only non-self-intersecting elements
and let $\xdag\in\Mcurves_0$. 
Suppose that $F \colon \Mcurves_0\to \Yspace$ is weakly 
sequentially continuous (with respect to the topologies of $\Xspace$ and $\Yspace$) 
and injective and $\Mcurves_0$ is weakly closed in the case of \eqref{eq:defi_tikh}. 
\begin{enumerate}
\item \emph{(existence)}
The infimum of the Tikhonov functionals 
in \eqref{eq:defi_tikh} and \eqref{eq:defi_tikh_with_penetr} 
is attained for any $\alpha>0$.
\item \emph{(regularizing property)} 
Suppose that $F$ is  injective. 
Moreover, consider a sequence of data $(\data^{\delta_n})$ with 
$\|\data^{\delta_n}-F(\xdag)\|\leq \delta_n\to 0$ 
as $n\to\infty$. 
Assume that the regularization parameters are chosen such that
\[
\alpha_n\to 0 \qquad  \mbox{and}\qquad 
\frac{\delta_n}{\sqrt{\alpha_n}}\to 0\,.
\]
Then for any sequence of minimizers of the Tikhonov functionals we have
\begin{align}
\label{eq:further_conv}
&\lim_{n\to\infty} \bignorm{\curve_{\alpha_n}^{\delta_n}-\xdag}_{\Xspace} 
= \lim_{n\to\infty} \bignorm{\gamma_{\curve_{\alpha_n}^{\delta_n}}-\gamma_{\xdag} }_\infty
= 0,\\
\label{eq:conv_Y}
	&\lim_{n\to\infty} \bignorm{
		F\bigparen{\curve_{\alpha_n}^{\delta_n}}- F\bigparen{\xdag}
	}_{\Yspace} = 0.
\end{align}
\item \emph{(convergence rates)}
Suppose in the case of \eqref{eq:defi_tikh} that 
there exists a loss function $l \colon \Mcurves\times\Mcurves \to [0,\infty)$ 
and a concave, 
increasing function $\varphi \colon [0,\infty)\to [0,\infty)$ with $\varphi(0)=0$ 
such that $\xdag$ satisfies the variational source condition 
\begin{equation}\label{eq:vsc}
l(\curve,\xdag) \leq \energy(\curve,\curve_*)-\energy(\xdag,\curve_*) 
+ \varphi \bigparen{\|F(\curve)-F(\xdag)\|_{\Yspace}^2}
\end{equation}
for all $\curve\in \Mcurves_0$. Then the reconstruction error for an optimal 
choice $\overline{\alpha}$ of $\alpha$ is bounded by
\begin{equation}\label{eq:conv_rate}
l\bigparen{\curve_{\overline{\alpha}}^{\delta},\xdag}\leq 2 \, \varphi(\delta^2).
\end{equation}
\end{enumerate}
\end{theorem}

\begin{proof}
We define a functional $\totEnerg \colon \Xspace\to [0,\infty)$ by
\begin{align*}
&\totEnerg(\curve) \ceq \begin{cases}
\energy(\curve,\curve_*),&\mbox{if }\curve\in\Mcurves_0,\\
\infty,&\mbox{else}
\end{cases}
\quad\mbox{or}\quad\\
&\totEnerg(\curve) \ceq \begin{cases}
\energy(\curve,\curve_*)+\penetrEnerg(\curve),&\mbox{if }\curve\in\Mcurves,\\
\infty,&\mbox{else}
\end{cases}
\end{align*}
in the case of \eqref{eq:defi_tikh} or \eqref{eq:defi_tikh_with_penetr}, respectively. 
We show that in both cases $\totEnerg$ is weakly sequentially lower semi-compact, 
i.e.\ sublevel-sets of $\totEnerg$ are weakly sequentially compact. 
In the first case this follows from Proposition \ref{prop:energy}, 
part~(\ref{it:energy:Elscompact}) and the assumption that $\Mcurves_0$ is weakly 
sequentially closed. In the second case this is a straightforward consequence 
of  Proposition~\ref{prop:energy}, part~(\ref{it:energy:Elscompact}) and 
Lemma \ref{lem:lsc_penetrEnerg}.

Extending $F$ to an operator $\widetilde{F} \colon \Xspace\to\Yspace$ 
in an arbitrary fashion, we can formally write the Tikhonov regularizations 
\eqref{eq:defi_tikh} and \eqref{eq:defi_tikh_with_penetr} 
as a minimization problem over $\Xspace$, 
\[
\curve_{\alpha}^{\delta}\in \argmin_{\curve\in\Xspace}
\bracket{\|\widetilde{F}(\curve)-\data^{\delta}\|_{\Yspace}^2 
+ \alpha \, \totEnerg(\curve)}.
\]
and apply standard convergence results for generalized Tikhonov regularization. 
The first statement now follows from \cite[Theorems 3.22]{Scherzer_etal:09} or 
\cite[Theorems 3.2]{flemming:12b}. 

To prove the second statement, 
let $\xdag=(\Angle^\dagger,L^{\dagger},\BasePoint^\dagger)$ and 
$\curve_{\alpha_n}^{\delta_n} = (\Angle_n,L_n,\BasePoint_{n})$ and 
recall from \cite[Theorem 3.26]{Scherzer_etal:09} or \cite[Theorem 3.4]{flemming:12b} that \eqref{eq:conv_Y} holds true, and 
for an injective operator we have 
weak convergence of $\curve_{\alpha_n}^{\delta_n}$ to $\xdag$ as well as 
$\lim_{n\to\infty}\totEnerg(\curve_{\alpha_n}^{\delta_n}) = \totEnerg(\xdag)$. 
Since $\energy$ and $\penetrEnerg$ are both weakly sequentially lower semicontinuous 
it follows that 
$\lim_{n\to\infty}\|\Angle_n'-\Angle_*'\|_{L^2}^2 = \|(\Angle^{\dagger}-\Angle_*)'\|_{L^2}^2$.
This implies 
\begin{align*}
\|(\Angle_n-\Angle^{\dagger})'\|_{L^2}^2 
&= \|(\Angle_n-\Angle_*)'\|_{L^2}^2 -\|(\Angle^{\dagger}-\Angle_*)'\|_{L^2}^2
+ \ninnerprod{(\Angle^{\dagger}-\Angle_*)',(\Angle_n-\Angle^{\dagger})'
}_{L^2}\\
&\to 0 \qquad \mbox{as }n\to \infty.
\end{align*}
Modulo shifts in $2 \, \pi \, \Z$, we may assume that $\Angle_n(0) \in \intervalcc{- \pi , \pi}$.
By passing to a subsequence, we may assume that $\Angle_{n_k}(0) \to \Angle^\dagger(0)$.
Using the equivalent norm $\|\Angle\|_* \ceq \nabs{\Angle(0)} + \nnorm{\Angle'}_{L^2}$ 
on $H^1([0,1],\Rset)$ this yields strong convergence of $(\Angle_{n_k})$ 
to $\Angle^{\dagger}$ in $H^1([0,1],\Rset)$. As weak convergence in $\Rset^2$ 
is equivalent to strong convergence, $(L_{n_k},\BasePoint_{n_k})$ also 
converges strongly to  $(L^{\dagger},\BasePoint^{\dagger})$.   
As this holds true for any subsequence, the whole sequence 
$(\curve_{\alpha_n}^{\delta_n})$  converges strongly 
to $\xdag$ in $\Xspace$. This implies strong convergence of the corresponding 
curves in the supremum norm. 

The third statement 
follows from \cite{grasmair:10} or \cite[Theorem 4.11]{flemming:12b}. 
\end{proof}



We point out that the variational source condition \eqref{eq:vsc} is related 
to stability results as worked out for inverse medium scattering problems in 
\cite{HW:15} where such conditions with logarithmic functions $\varphi$ hold true 
under Sobolev smoothness conditions on the solution. However, for inverse obstacle 
scattering problems no such verifications of variational source conditions are known 
so far. 

\begin{remark}
It  can be seen from the references cited in the proof of Theorem \ref{theo:main} that the results can be extended to the case where $\Yspace$ 
is a Banach space and $\|F(\curve)-\data^{\delta}\|_{\Yspace}^2$ is replaced by 
more general data fidelity terms $\mathcal{S}(F(\curve),\data^{\delta})$. 
\end{remark}


\section{Inverse obstacle scattering problem}\label{sec:scat}

As a prominent example of an obstacle scattering problem we consider 
here the scattering of time-harmonic acoustic waves at a sound-soft cylindrical obstacle. 
The cross section of this obstacle is described by a bounded, connected, and 
simply connected H\"older $C^{1,\alpha}$-smooth domain $\varOmega_{\mathrm{int}}$ 
($\alpha>0$). Then its 
unbounded complement $\varOmega\ceq \Rset^2\setminus \overline{\varOmega_{\mathrm{int}}}$ 
is connected,  and the boundary curve will be denoted by 
$\varGamma=\partial\varOmega = \partial \varOmega_{\mathrm{int}}$. 
Further we consider an incident plane wave 
$u_{\mathrm{i}}(x)=\exp(\ii \, k \, \langle x,d\rangle)$ with wavenumber $k>0$ and  direction $d\in \Sb$. 
Then the forward problem consists in finding a scattered wave 
$u_{\mathrm{s}}\in H^2_{\mathrm{loc}}(\varOmega)$ such that the total wave $u\ceq u_{\mathrm{i}} + u_{\mathrm{s}}$ solves the Helmholtz equation with Dirichlet boundary condition
\begin{subequations}\label{eqs:dirichlet_bdy_problem}
\begin{align}
\Delta u(x) + k^2 u(x) = 0,\qquad  &x\in \varOmega,\label{eq:Helmholz}\\
u(x) = 0 ,\qquad &x\in \varGamma, \label{eq:DirichletCond}\\
\intertext{together with the Sommerfield radiation condition}
 \lim_{\absval{x}\to 0} \sqrt{\absval{x}} \paren{\frac{\partial u_{\mathrm{s}}(x)}{\partial \absval{x}} - \ii \, k \, u_{\mathrm{s}}(x) }& = 0 \label{eq:SRC}
\end{align}
\end{subequations}
which holds uniformly for all directions $\frac{x}{\absval{x}}\in \Sb$. This problem is well-posed under the above conditions (see e.g.\ \cite{mclean:00}) and can for example 
be solved using boundary integral equations (see \cite{CK:12}). 
Recall that solutions to the Helmholtz equation which satisfy the Sommerfield radiation condition (\ref{eq:SRC}) have the asymptotic behavior
\begin{equation}\label{eq:waveFarField}
u_{\mathrm{s}}(x) = \frac{\mathrm{e}^{\mathrm{i} k \absval{x}}}{\sqrt{\absval{x}}} \paren{u_{\infty}\paren{\tfrac{x}{\absval{x}},d} + \mathcal{O} \paren{\tfrac{1}{\absval{x}} }}, \qquad \absval{x}\to \infty
\end{equation}
(see \cite[Sect.\ 2.2 and 3.4]{CK:12}). 
The function $u_{\infty}(\cdot,d)$ is analytic on $\Sb$ and known as the far field 
pattern of the scattered wave $u_{\mathrm{s}}$. Often the far field pattern 
$u_\infty\in L^2(\Sb\times\Sb)$ can only be measured on some submanifold 
$\mathbb{M}\subset \Sb\times\Sb$, e.g.\ $\mathbb{M}=\Sb\times\{d\}$ for one 
incident field or $\mathbb{M}= \{(d,-d)\colon d\in\Sb\}$ for backscattering data.

With the definitions of Section~\ref{sec:elastic_energy}
we may describe the inverse problems as operator equations on the Riemannian 
manifold $\Mcurves$: We introduce the operator $F \colon \Mcurves\to L^2(\mathbb{M})$ 
mapping $\curve\in\Mcurves$ to the far field pattern $\uinf$ of 
the scattered field in problem (\ref{eqs:dirichlet_bdy_problem})
for the domain $\varOmega$ corresponding to $\curve$. More precisely, the boundary $\varGamma$ is given by the image of the curve parameterization $\gamma_{\curve}(\Sb)$ and 
$\varOmega$ is the unbounded component of $\Rset^2\setminus \gamma_{\curve}(\Sb)$. 
The inverse problem is described by the operator equation
\begin{equation}\label{eq:defi_F_scat}
F(\curve) = \uinf.
\end{equation}

By Schiffer's uniqueness result (\cite[Theorem 5.1]{CK:12}) $F$ is injective 
if $\mathbb{M}=\Sb\times \Sb$, and by the uniqueness result of Colton and Sleeman 
(\cite[Theorem 5.1]{CK:12}) it is also injective if $\mathbb{M}$ is the product 
of $\Sb$ with some finite set and if all curves $\gamma_\curve$ for 
$\curve\in\Mcurves_0$ are contained in a ball of a certain size. (Both results are 
stated in \cite{CK:12} for $\Rset^3$, but also hold true in $\Rset^2$.) 

Let us show that the operator $F$ also satisfies the remaining assumptions of 
Theorem \ref{theo:main}:

\begin{proposition}\label{prop:continuityF}
The operator $F$ 
maps weakly convergent sequences in $\Mcurves_0$ (with respect to the topology of 
$\Xspace$) to strongly convergent sequences in $L^2(\Sb)$ and is continuously 
Fr\'echet differentiable. \\
In particular, $F$ is strongly and weakly continuous. 
\end{proposition}

\begin{proof}
Notice that the linear mapping $\Xspace\to C^1(\Sb,\Rset^2)$, 
$\curve\mapsto \gamma_{\curve}$ defined by~\eqref{eq:CurveRepresentation} 
is compact by embedding theorems for Sobolev spaces, and hence it 
maps weakly convergent sequences to strongly convergent sequences. 
Moreover, the forward scattering operator 
$C^1(\Sb,\Rset^2)\to L^2(\Sb)$, $\gamma_{\curve}\mapsto u_{\infty}$  
is continuously Fr\'echet differentiable, and in particular continuous 
by \cite[Theorem 1.9]{Hohage:99}. Therefore, the composition of these 
two mappings is continously Fr\'echet differentiable and maps weakly convergent 
to strongly convergent sequences. 
\end{proof}

Notice that by the last proposition the operator equation (\ref{eq:defi_F_scat}) 
on an infinite dimensional manifold $\Mcurves_0$ is ill-posed in the sense 
that there cannot exist a strongly continuous inverse of $F$. 
(Otherwise every weakly convergent sequence 
in $\Mcurves_0$ would be strongly convergent.) 
This implies the need for regularization to solve this equation.

\begin{remark}[translation invariance for phaseless data]
In many applications only the squared amplitude of the far field can be measured, but 
not the phase. As the amplitude of the far field is translation invariant (see \cite{kress:97}), the corresponding forward operator $\Fphase(\curve)\ceq|F(\curve)|^2$ 
is also translation invariant, i.e.\ $\Fphase(\Angle, L, \BasePoint)$ does not depend on the base point $\BasePoint$. This case fits very well into our setting since 
the shape manifold may simply be reduced to 
$\McurvesPhase\ceq\AngleSpace \times [L_1,L_2]$. 
\end{remark}


\section{Discrete setting}\label{sec:discrete}

In order to treat bending energy computationally, we represent closed curves by closed polygons. To this end, consider an arbitrary (but fixed) partition 
$(0=\tau_{0}<\tau_1 < \dots < \tau_n=1)$ of the unit interval and let the angle variable be given by a \emph{piecewise constant} function represented by 
a vector $\dAngle = (\Angle_1, \dots, \Angle_n)$, i.e.\ $\Angle(t)=\Angle_j$ 
for $t\in(\Angle_{j-1},\Angle_j]$. In perfect analogy to  \eqref{eq:CurveRepresentation}, we then define a polygon of length $L$ by
\begin{align}\label{eq:DiscCurveReconstruct}
	\gamma(t) \ceq \BasePoint +L \int_0^t \, \bigparen{\!\cos(\Angle(\tau)),\sin(\Angle(\tau))} \, \dd \tau.
\end{align}
Analogously to the smooth case, in order to fulfill the closing conditions \eqref{eq:ClosingConditions}, $\dAngle$ needs to satisfy
\begin{align}\label{eq:DiscCurveConstraints}
\dConstr(\dAngle)= 0 , \quad \mbox{where}\quad 
\dConstr(\dAngle) = \int_0^1 \bigparen{\! \cos (\theta(t)), \sin (\theta(t))} \, \dd t.
\end{align}%
Define the \emph{turning angles} by $[\dAngle]_{i} \ceq (\Angle_{i+1}-\Angle_i)$, where indices are taken modulo $n$ and  $[\dAngle]_{i}$ is shifted such that  $[\dAngle]_{i} \in (-\pi, \pi]$ for all $i$. The number $\left(\sum_i [\dAngle]_{i} \right)/ 2\pi$ is known as the  \emph{discrete turning number} of $\gamma$.

Let $\dAngleSpace\ceq\{\dAngle\in\Rset^{n}\,|\,\dConstr(\dAngle)=0\}$ and 
define the space of discrete curves  by
\begin{align*}
	\dMcurves \ceq \dAngleSpace \times [L_1,L_2] \times B \quad \subset \quad \dXspace \ceq \R^{n} \times \R \times \R^2,
\end{align*}
for a  compact, convex set of base points $B \subset \R^2$ and minimal and maximal acceptable curve lengths $0 < L_1 \leq L_2 < \infty$. On this space, the scale-invariant version of discrete bending energy for a  curve $\dcurve \in \dMcurves$ is readily defined as 
\begin{align}\label{eq:defi_disc_energy}
\dEnergy(\dcurve) \ceq \sum_{i=1}^n \left(\frac{[\dAngle]_{i}}{h_{i}}\right)^2 h_{i} \, =\, \sum_{i=1}^n \frac{([\dAngle]_{i})^2}{h_{i}},
\end{align}
see, e.g., \cite{hencky1921angenaherte}. Here the \emph{dual edge lengths} are given by $h_{i} \ceq (\tau_{i+1}-\tau_{i-1})/2$ for $i \in \{1,\dots, n\}$, where we set 
$\tau_{n+1}=1+\tau_1$. 
This expression provides the natural analogue\footnote{Notice that discrete bending energy corresponds to its smooth counterpart in the sense that turning angles at vertices correspond to curvatures \emph{integrated} over dual edges, i.e., $[\dAngle]_{i} \cong \int_{(\tau_i+\tau_{i-1})/2}^{(\tau_{i+1}+\tau_{i})/2}\kappa(s) \, \dd s$. This perspective naturally leads to formulation \eqref{eq:defi_disc_energy}.} of the smooth version \eqref{eq:defi_energy}. It goes back to the work of Hencky in his 1921 PhD thesis \cite{hencky1921angenaherte} and is in the spirit of discontinuous Galerkin (DG) methods (\cite{MR1885715}).  
A completely analogous discrete version of this energy can be defined for \emph{open} polygons. In this case, for clamped boundary conditions and under the constraint of fixed total curve length, the set of minimizers of this discrete energy converges in Hausdorff distance to the corresponding set of smooth minimizers under mesh refinement, see \cite{SchuWa19}. More specifically, the angle variables converge in $L^\infty$ and in $W^{1,p}$ for $p\in [2, \infty)$ under a suitable smoothing operator for the angle variables. Finally, a discrete analogue $\dEnergy(\dcurve,\dcurve_*)$
of the smooth pre-curved energy $\energy(\curve,\curve_*)$ is readily obtained by replacing $[\dAngle]$ by $([\dAngle]-[\dAngle]_*)$ in~\eqref{eq:defi_disc_energy}. 

\paragraph*{Implementation details} For convenience, we briefly sketch here the implementation of our method. 
The regularized functional that we seek to  minimize on the space $\dMcurves \subset \dXspace $ is of the form 
\begin{align}\label{eq:DiscRegFunc}
	\cJ^{\alpha} \colon \dMcurves \to \R,
	\quad
	\dcurve
	\mapsto
	\tfrac{1}{2}\bignorm{\discr{F}(\dcurve)-\dy^\delta}_{\dYspace}^2
	+ \alpha \,  \dTotEnerg(\dcurve) .
\end{align}
Here, $\discr{F}:\dMcurves\to\dYspace$ is some discretization for polygonal closed curves of the forward operator $F$, the term $\dy^\delta\in\dYspace$  represents the measured data in some finite dimensional Euclidean space $\dYspace$, the scalar $\alpha \geq 0$ is the regularization parameter,  and $\dTotEnerg=\dEnergy$ or $\dTotEnerg=\dEnergy+\dPenetrEnerg$ with a discrete approximation 
$\dPenetrEnerg$ of the M\"obius energy $\penetrEnerg$. 

\begin{remark}
We skip the requisite details on the definition of $\dPenetrEnerg$  since our numerical experiments show that in practice the tracking term $\tfrac{1}{2}\bignorm{\discr{F}(\dcurve)-\dy^\delta}_{\dYspace}^2$ (see \eqref{eq:DiscRegFunc} below) is sufficient to prevent iterates from developing self-in\-ter\-sec\-tions.
Notwithstanding, for details on discrete M\"{o}bius energy, see \cite{MR1470748,MR1702037}, and for
\hbox{$\varGamma$-convergence} to the smooth case see \cite{MR3268981}.
\end{remark}

The discrete nonlinear Tikhonov regularization on $\dMcurves$ may then be written as
the following constrained minimization problem: 
\begin{align}\label{eq:tikh_constraints}
	\text{Minimize} 
	\;\;
	\cJ^{\alpha}(\dcurve)
	\;\;
	\text{subject to}
	\;\;
	\dConstr(\dcurve)= 0 
	\;\;
	\text{and}
	\;\;
	 (L,\BasePoint) \in  [L_1,L_2] \times B.
\end{align}
We will ignore the inequality constraints 
$(L,\BasePoint) \in  [L_1,L_2] \times B$ for simplicity, although it would not be difficult to include 
them. In particular, these constraints never became active 
in our numerical experiments. We only require these constraints for the theoretical 
analysis in Section \ref{sec:tikhonov}.

Since $\discr{F}$ does not have a natural extension outside the discrete 
shape space $\dMcurves = \{\dcurve\,|\,\dConstr(\dcurve)=0\}$, standard methods 
of constrained nonlinear programming are not applicable.
When using iterative  methods for minimizing $\cJ^{\alpha}$, we require an \emph{intrinsic} stepping method on the constraint manifold $\dMcurves$ in order to supply the forward operator $\discr{F}$ with meaningful input. Prominent examples of such methods are intrinsic Newton-type algorithms on Riemannian manifolds, see, e.g., \cite{MR2968868}. In such methods, one determines the update direction $\du \in \dXspace$ by solving a saddle point system of the form
\begin{align}
\begin{pmatrix}
H(\dcurve) & D\dConstr^\top(\dcurve) \\
D\dConstr(\dcurve) & 0
\end{pmatrix} 
\begin{pmatrix}
\du \\ \mu
\end{pmatrix} = 
\begin{pmatrix}
-D \cJ^{\alpha}(\dcurve) \\ 0
\end{pmatrix} ,
	\label{eq:SaddlePointSystem}
\end{align}
where $H$ is (a surrogate for) the Hessian of the objective functional, the manifold $\dMcurves$ is given by the constraint equations \eqref{eq:DiscCurveConstraints}, which we encode by a function $\dConstr:\dXspace \to \R^2$, and $\mu\in \Rset^2$ denotes a Lagrange multiplier. 
The resulting linear systems have roughly the size $n\times n$ and can be solved using a direct 
solver. In our implementation, we usually use $n= 100$.

A first example is the full intrinsic Hessian, which can be obtained from the Lagrange function 
$\mathcal{L}(\dcurve,\lambda)\ceq \cJ^{\alpha}(\dcurve) + \lambda^{\top}\dConstr(\dcurve)$ 
of \eqref{eq:tikh_constraints} as
\begin{align}\label{eq:fullHessian}
H(\dcurve) = D_{\dcurve}^2 \mathcal{L}(\dcurve,\lambda_{\dcurve})
= D^2 \cJ^{\alpha}(\dcurve)  + \lambda_{\dcurve}^{\top} D^2 \dConstr(\dcurve) .
\end{align}
The requisite Lagrange multiplier $\lambda_{\dcurve}^{\top}\in \Rset^2$ 
is obtained by multiplying the equation $D_{\dcurve}\mathcal{L}(\dcurve,\lambda)=0$  
by $D \dConstr^\dagger(\dcurve)$ from the right, i.e.,
\begin{align*}
\lambda_{\dcurve}^{\top} = - D\cJ^{\alpha}(\dcurve) \, D \dConstr^\dagger(\dcurve) .
\end{align*}
Here $D \dConstr^\dagger(\dcurve) $ denotes the Moore-Penrose inverse with respect to a 
finite difference approximation of the $H^1$-inner product. 

Notice that assembling the system with the full intrinsic Hessian contains a contribution of the form $\ninnerprod{\discr{F}(\dcurve)-\dy^\delta, D^2\discr{F}(\dcurve) ( \cdot, \cdot)}_{\dYspace}$, which is dense and costly to compute. We therefore use a Gau\ss-Newton inspired surrogate, which is given in bilinear form as\footnote{Notice that in this formulation we have also dropped the additional term of the form 
$\ninnerprod{\discr{F}(\dcurve)-\dy^{\delta},D\discr{F}(\dcurve) \, D\dConstr^\dagger(\dcurve) \,  D^2\dConstr(\dcurve)}_{\dYspace}$  since it does not lead to improved convergence rates.}
\begin{align}\label{eq:GaussNewton}
H(\dcurve) = \ninnerprod{D\discr{F}(\dcurve) \, \cdot \,, D\discr{F}(\dcurve) \, \cdot \, }_{\dYspace} + \alpha \Hess \dTotEnerg(\dcurve),
\end{align}
where we identify matrices with bilinear forms and where
the \emph{intrinsic} energy Hessian has the form 
\begin{align}\label{eq:SurrogateHess}
\Hess \dTotEnerg(\dcurve) = D^2 \dTotEnerg(\dcurve) - D\dTotEnerg (\dcurve) \,D\dConstr^\dagger(\dcurve) D^2\dConstr(\dcurve).
\end{align}
Notice that the second term on the right hand side of this equation arises from 
the second term on the right hand side of \eqref{eq:fullHessian}.
In the language of differential geometry, the term $D\dConstr^\dagger(\dcurve) \, D^2\dConstr(\dcurve)$ encodes the \emph{second fundamental form} of the constraint manifold. 
The quantities on the right hand side of  \eqref{eq:SurrogateHess} are easy to assemble 
for $\dTotEnerg=\dEnergy$ due to the quadratic nature of~$\dEnergy$.

Another attractive alternative is to use
\begin{align*}
H(\dcurve) = \ninnerprod{D\discr{F}(\dcurve) \, \cdot \,, D\discr{F}(\dcurve) \, \cdot \, }_{\dYspace} + \alpha \, \ninnerprod{ \cdot, \cdot}_{\Xspace}.
\end{align*}
This way, $H(\dcurve)$ is always positive definite on the null space of $D\dConstr(\dcurve)$ and the saddle-point matrix from  \eqref{eq:SaddlePointSystem} is guaranteed to be continuously invertible.
Thus, in this case, the method boils down to a gradient descent in the manifold $\dMcurves$ with respect to the Riemannian metric induced by $H$.

Once an update direction $\du$ has been computed in the above fashion, the next iterate is found by first setting $\dx_0 = \dcurve + t \, \du$ for some small $t>0$. Restoring feasibility (i.e., ensuring that the next iterate resides on the constraint manifold) is then achieved by iterating 
\begin{align}\label{eq:ProjectOnConstraint}
\dx_{k+1} = \dx_k - D\dConstr^\dagger(\dx_k) \, \dConstr(\dx_k),
\end{align}
until $\dConstr(\dx_k)$ is sufficiently small.\footnote{Notice that the Newton-type method \eqref{eq:ProjectOnConstraint} for underdetermined systems would correspond to a nearest point projection if the constraint were linear.} The step size $t$ can be determined by a standard backtracking line search, while the matrix-vector product $\tilde \du = D\dConstr^\dagger(\dx) \, \tilde \dv$ is computed by solving the saddle point problem  
\begin{align*}
\begin{pmatrix}
	G_{\dXspace} & D\dConstr^{\top}(\dx) \\
D\dConstr(\dx) & 0
\end{pmatrix} 
\begin{pmatrix}
\tilde \du \\ \tilde\mu
\end{pmatrix} = 
\begin{pmatrix}
0 \\ \tilde \dv
\end{pmatrix} .
\end{align*}
Here $G_{\dXspace}$ is the Gram matrix of the discrete $H^1$-inner product on $\dXspace$, the 
upper left $n\times n$~block of which is a finite-difference Laplacian. 
Analoguously, $D\dEnergy(\dcurve) \,  D \dConstr^\dagger(\dcurve) 
=  (D \dConstr^\dagger(\dcurve))^{\top}  D\dEnergy(\dcurve)$ 
can be computed this way by utilizing the dual saddle point system.
Finally, one updates $\dcurve$ to the last iterate~$\dx_k$.

\section{Ab-initio reconstructions}\label{sec:results}

In this section we demonstrate the benefits of our new approach in 
numerical experiments for the  inverse obstacle scattering problem 
introduced in Section~\ref{sec:scat}. The forward scattering problems 
were solved by a boundary integral equation method using a Nystr\"om 
method with $n$ points as described in \cite[sec. 3.6]{CK:12}. 
To this end we interpolated the polygonal curve approximations  
described in section \ref{sec:discrete} trigonometrically. 
Both the evaluation of discrete forward operator $\discr{F}$ and 
the evaluation of its Jacobian $D\discr{F}$ as described e.g.\ in \cite{Hohage:99} 
require $\mathcal{O}(n^3)$ flops. 

We always use 20 equidistant incident plane waves and $n=100$ points for 
the reconstruction curves; the far field pattern is measured at 40 equidistant measurement directions. 
The wavelength is chosen of the same order of magnitude as the diameter of the obstacle. 
In all our examples we added independent, identically distributed, centered Gaussian 
random variables to the simulated far field data at each sampling point  
such that the relative noise level in the $l^2$-norm was $5\%$ 
(in Figures \ref{fig:rec_letter_s}, \ref{fig:rec_letter_m} and \ref{fig:rec_three_lobes})
or $1\%$ (in Figures \ref{fig:factorLevel_method_reconstruction} and
\ref{fig:sausage_rec_comparison}). 

\begin{figure}[ht]
\centering
\begin{subfigure}{.3\textwidth}
\includegraphics[width=\textwidth,clip,trim={130 5 130 10}]{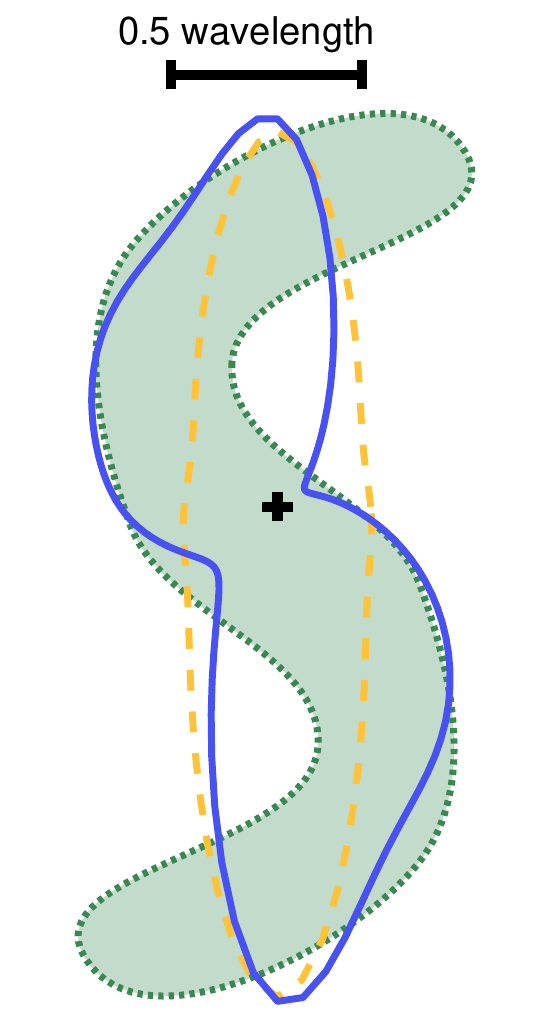}
\subcaption{ }
\end{subfigure}
\begin{subfigure}{.3\textwidth}
\includegraphics[width=\textwidth,clip,trim={130 5 130 10}]{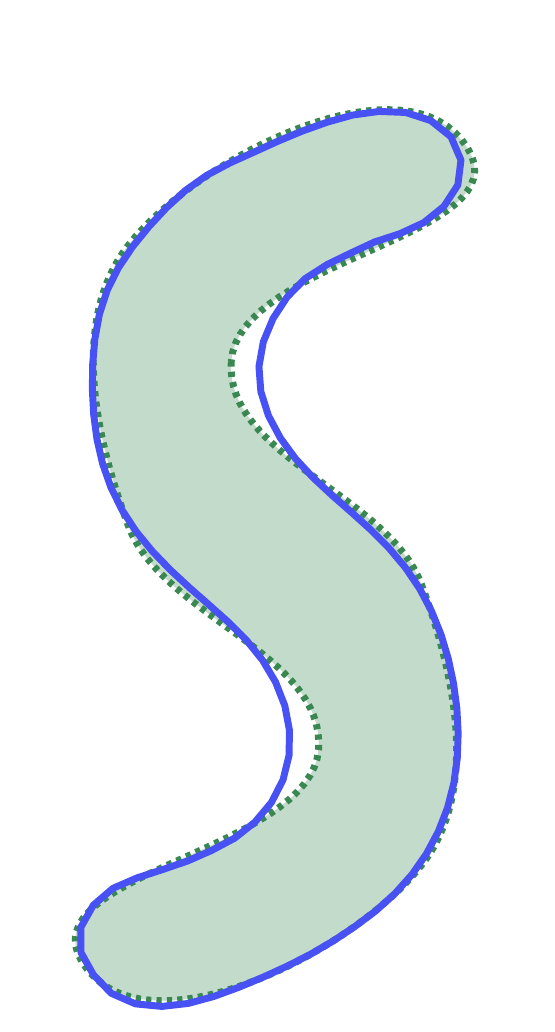}
\subcaption{ }
\end{subfigure}
\begin{minipage}{.37\textwidth}
\begin{subfigure}{1.\textwidth}
\includegraphics[width=\textwidth,clip,trim={0 20 0 0}]{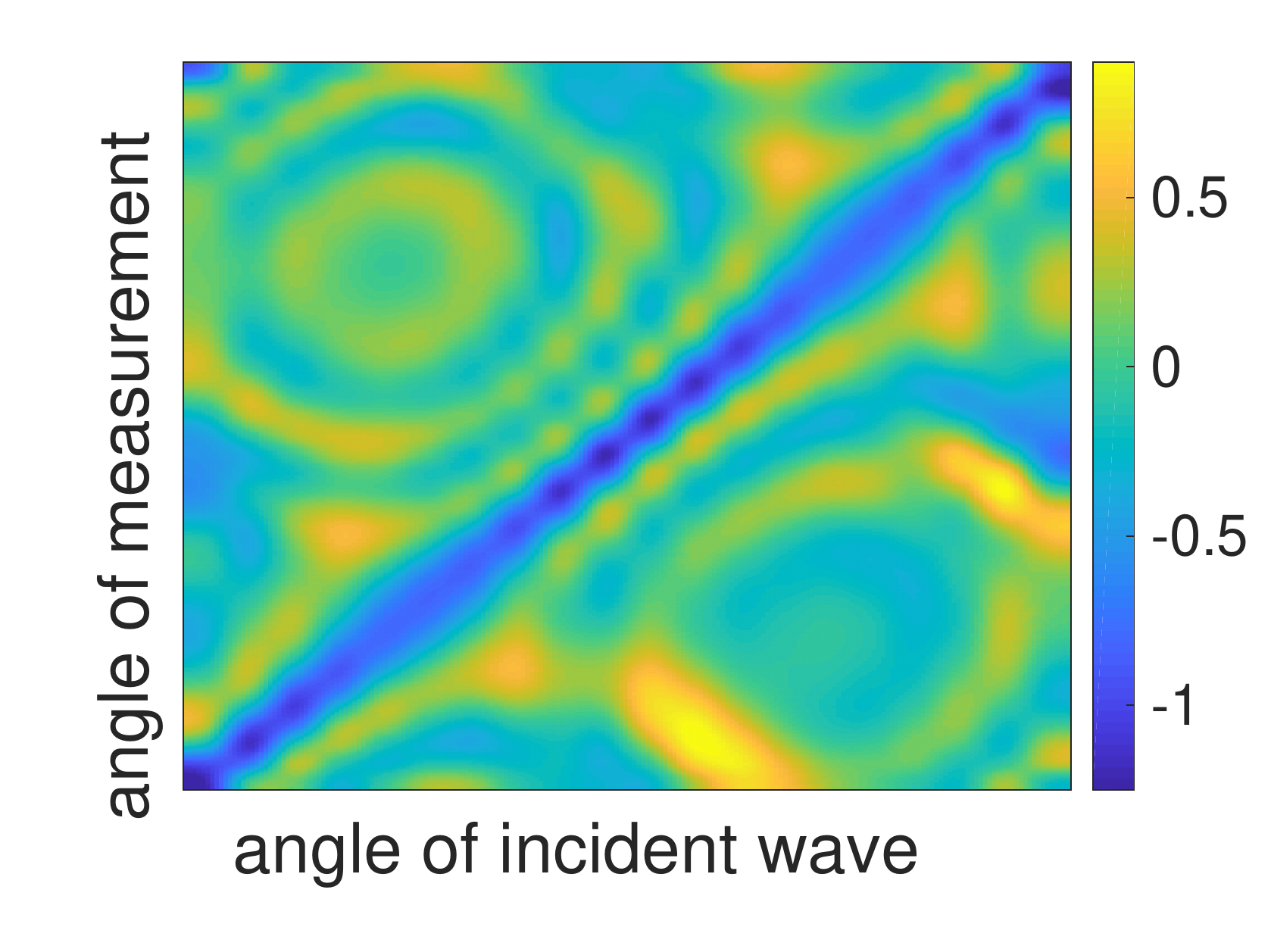}
\subcaption{ }
\end{subfigure}
\begin{subfigure}{1.\textwidth}
\includegraphics[width=\textwidth,clip,trim={0 20 0 0}]{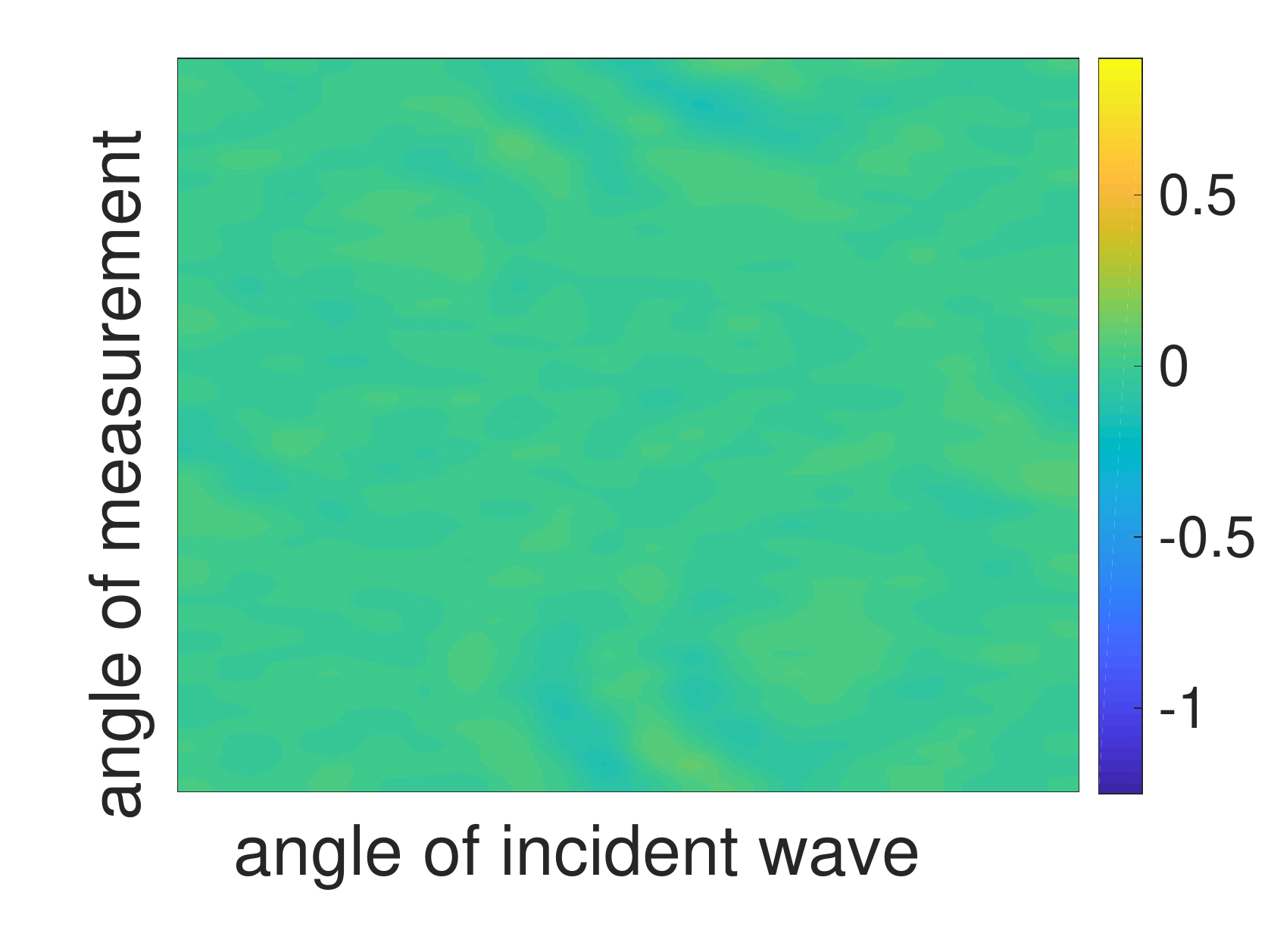}
\subcaption{ }
\end{subfigure}
\end{minipage}
\caption{\footnotesize Comparison of our method (b) to previous radial function 
parameterizations (a) for a smooth non-star-shaped domain.  
We use 20 equidistant incident waves with wavelength indicated in (a) 
and $5\%$ Gaussian white noise. 
(a) true obstacle (dotted green), initial guess (dashed yellow), and 
reconstruction (solid blue) using a radial function parameterization with center 
indicated by the black cross; 
(b) reconstruction using bending energy penalization; (c) real part of far field of the reconstruction in (b); (d) difference to the observations of the real parts of reconstructed far field in (b).}
\label{fig:rec_letter_s}
\end{figure}

The regularization parameter $\alpha$ was determined by the discrepancy principle. 
More precisely, we first minimized the Tikhonov functional for a large $\alpha$ 
by an intrinsic Gauss-Newton-type method as described in Section~\ref{sec:discrete} 
with update direction $\du$ defined by \eqref{eq:SaddlePointSystem}, 
\eqref{eq:GaussNewton}, \eqref{eq:SurrogateHess}. The Gauss-Newton iteration was 
stopped when $\|\du\|$ or the norm of the gradient of the Tikhonov functional 
$\|D \cJ^{\alpha}(\dcurve)\|$ was smaller than $10^{-5}$. 
Then we decreased $\alpha$ by a factor of~$2$ and minimized the Tikhonov functional for this smaller 
$\alpha$ using the previous minimizer as an initial guess as long as the 
condition $\norm{\discr{F}(\dcurve_{\alpha})-\dy^{\delta}}\geq 1.1\delta$ 
was satisfied. 

\begin{figure}[ht]
\centering
\begin{subfigure}{.41\textwidth}
\includegraphics[width=\textwidth,clip,trim={30 50 20 30}]{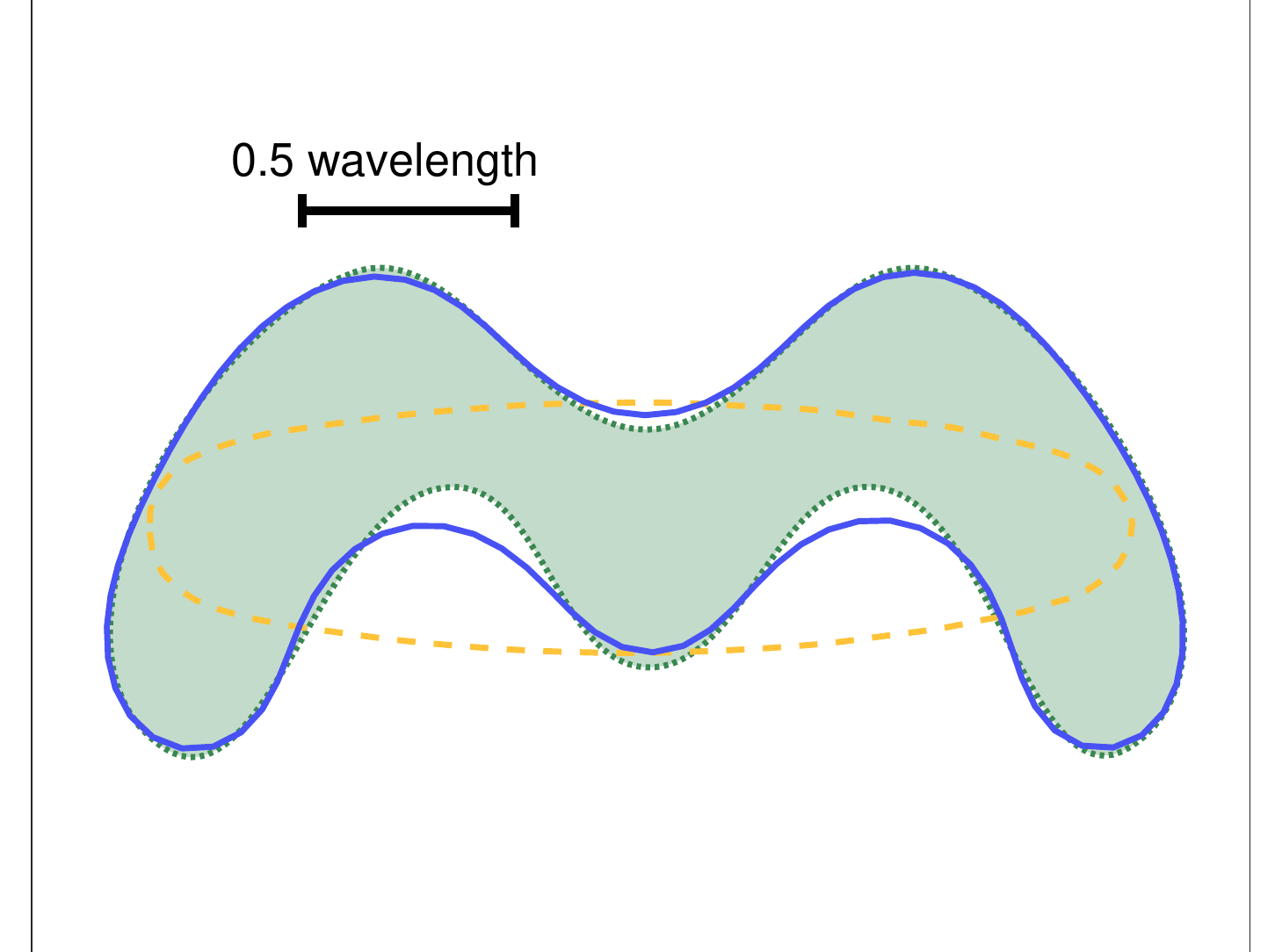}
\subcaption{ }
\end{subfigure}
\hspace*{2ex}
\begin{subfigure}{.41\textwidth}
\begin{tikzpicture}
\node[anchor=south west,inner sep=0] (image) at (0,0) {\includegraphics[width=\textwidth]{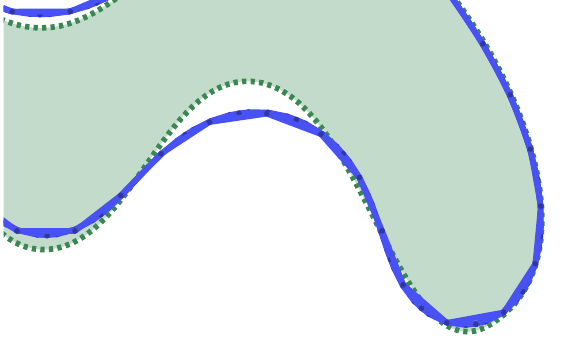}};
\begin{scope}[x={(image.south east)},y={(image.north west)}]
        \draw[black, thick] (0.,0.) rectangle (1.,1.);
    \end{scope}
\end{tikzpicture}
\subcaption{ }
\end{subfigure}
\caption{\footnotesize Reconstruction of a smooth non-star-shaped domain by our method  with $5\%$ Gaussian white noise. Parameters, line styles and colors are chosen as in 
Figure~\ref{fig:rec_letter_s}. 
Panel (b) shows a magnification of reconstructions for different numbers of points 
($n=50$, $100$, and $150$) illustrating the asymptotic independence
of the results on the choice of $n$. }
\label{fig:rec_letter_m}
\end{figure}

In Figures \ref{fig:rec_letter_s} and \ref{fig:rec_letter_m} 
we show reconstructions of two non-star-shaped domains. 
Figure \ref{fig:rec_letter_s} (d) illustrates that the far field pattern is uniformly 
fitted well. Moreover, we demonstrate in Figure \ref{fig:rec_letter_m} (b) that 
the reconstructions are almost independent of the choice of the number $n$ of points 
on the curves as long as $n$ is large enough. 
Also the number of Gau{\ss}-Newton steps and the regularization parameter $\alpha$ 
determined by the discrepancy principle do not depend on $n$. 
Note that concave parts of the boundary where multiple reflections occur in a geometrical 
optics approximation are more difficult to reconstruct than 
convex parts. 
In view of the fact that we use only one wave length which is almost of the size of 
the obstacle and a noise level of $5\%$, these reconstructions for this 
exponentially ill-posed problem are already remarkably good. 
The reconstructions could be further improved by using shorter wave lengths 
as illustrated in Figure \ref{fig:sausage_rec_comparison} (c).


\begin{figure}[ht]
\centering
\begin{subfigure}[t]{.31\textwidth}
\includegraphics[width=\textwidth,clip,trim={20 5 20 10}]{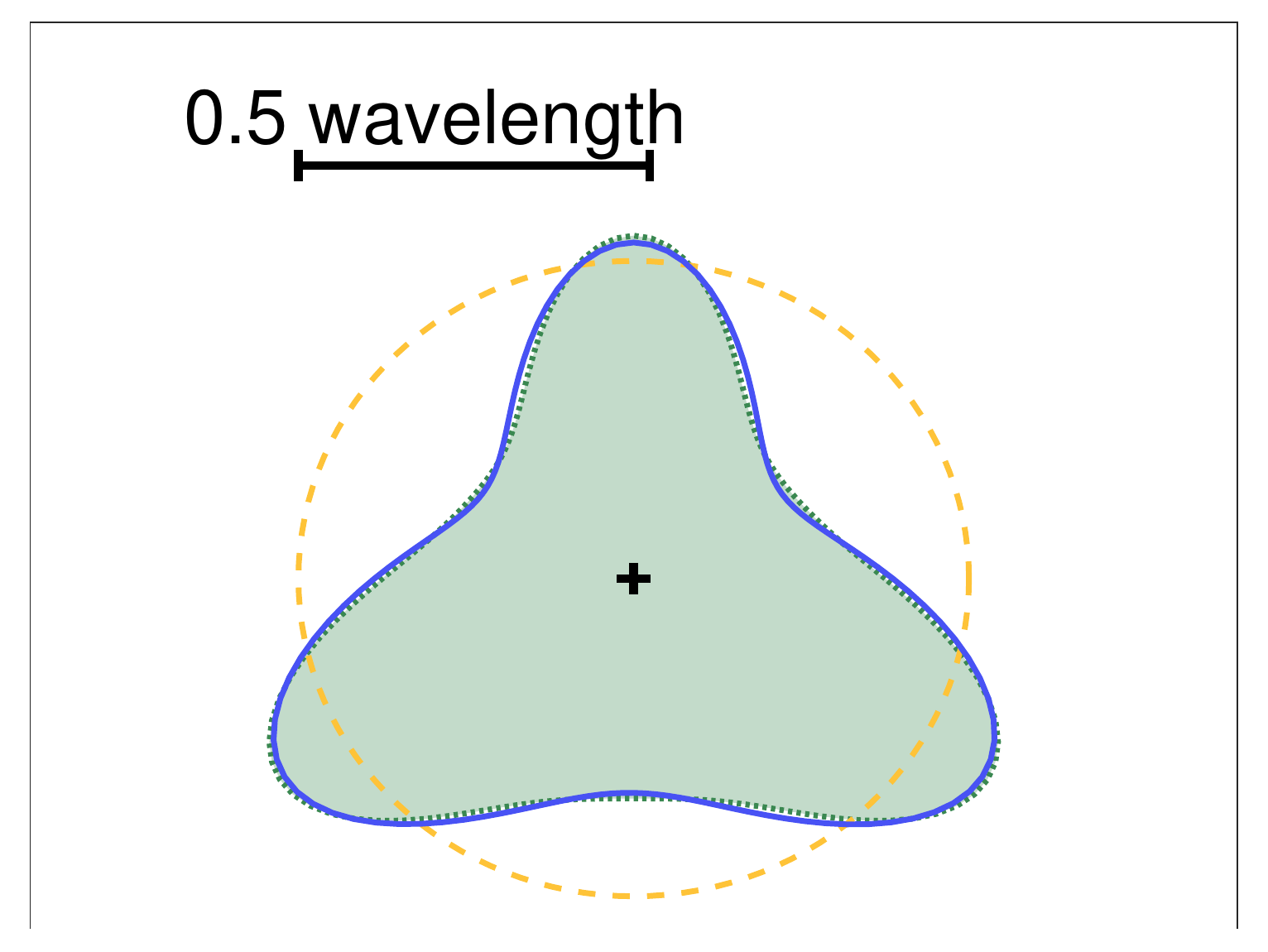}
\subcaption{ }
\end{subfigure}
\begin{subfigure}[t]{.31\textwidth}
\begin{tikzpicture}
\node[anchor=south west,inner sep=0] (image) at (0,0) {\includegraphics[width=\textwidth,clip,trim={20 5 20 10}]{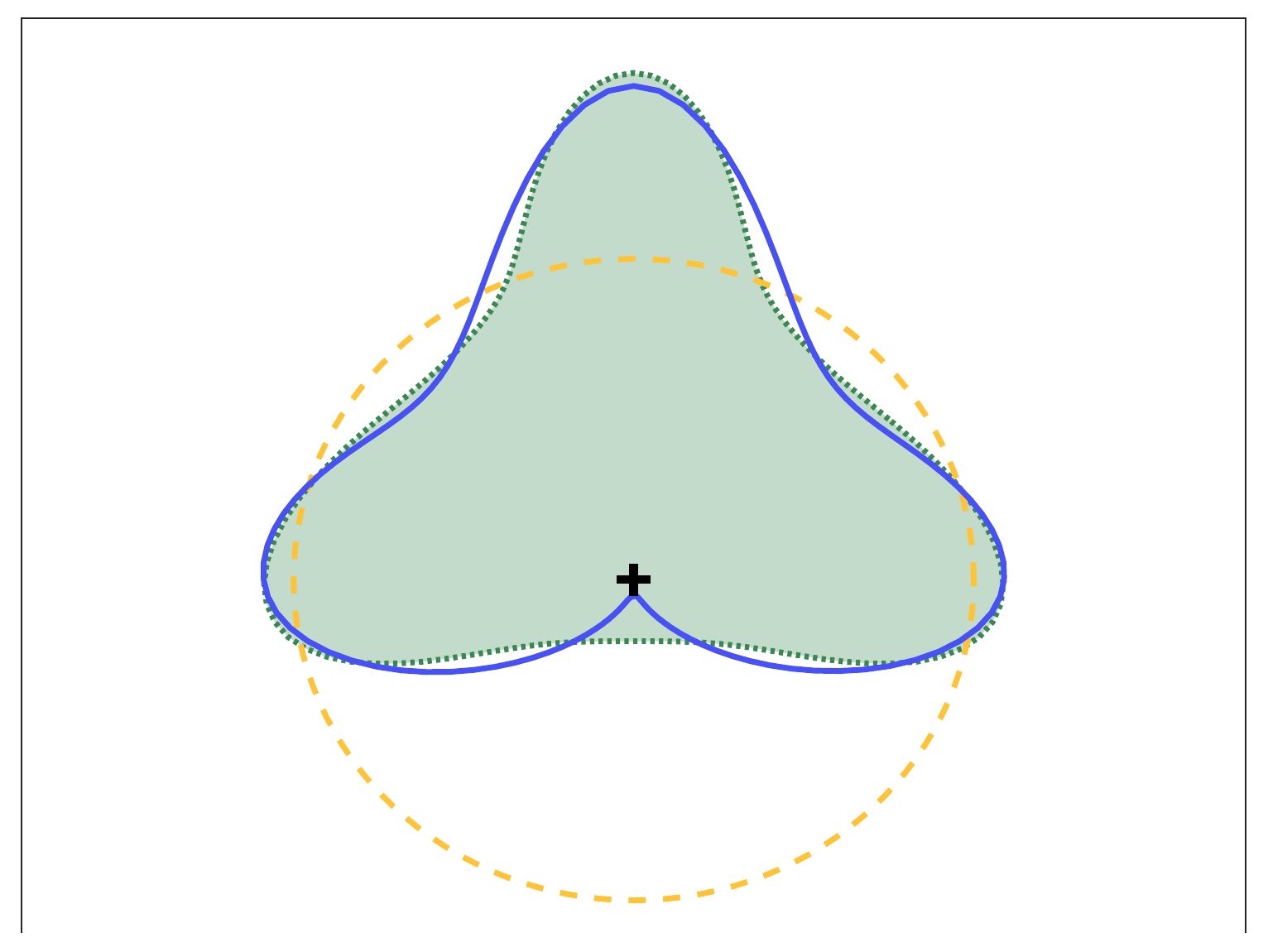}};
\begin{scope}[x={(image.south east)},y={(image.north west)}]
        \draw[black, thick] (0.32,0.2) rectangle (0.68,0.5);
    \end{scope}
\end{tikzpicture}
\subcaption{ }
\end{subfigure}
\begin{subfigure}[t]{.31\textwidth}
\begin{tikzpicture}
\node[anchor=south west,inner sep=0] (image) at (0,0) {\includegraphics[width=\textwidth,clip,trim={10 5 10 5}]{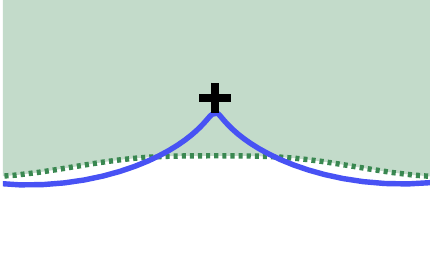}};
\begin{scope}[x={(image.south east)},y={(image.north west)}]
        \draw[black, thick] (0.,0.) rectangle (1.,1.);
    \end{scope}
\end{tikzpicture}
\subcaption{ }
\end{subfigure}
\begin{subfigure}[t]{.31\textwidth}
\begin{tikzpicture}
\node[anchor=south west,inner sep=0] (image) at (0,0) {\includegraphics[width=\textwidth,clip,trim={20 5 20 10}]{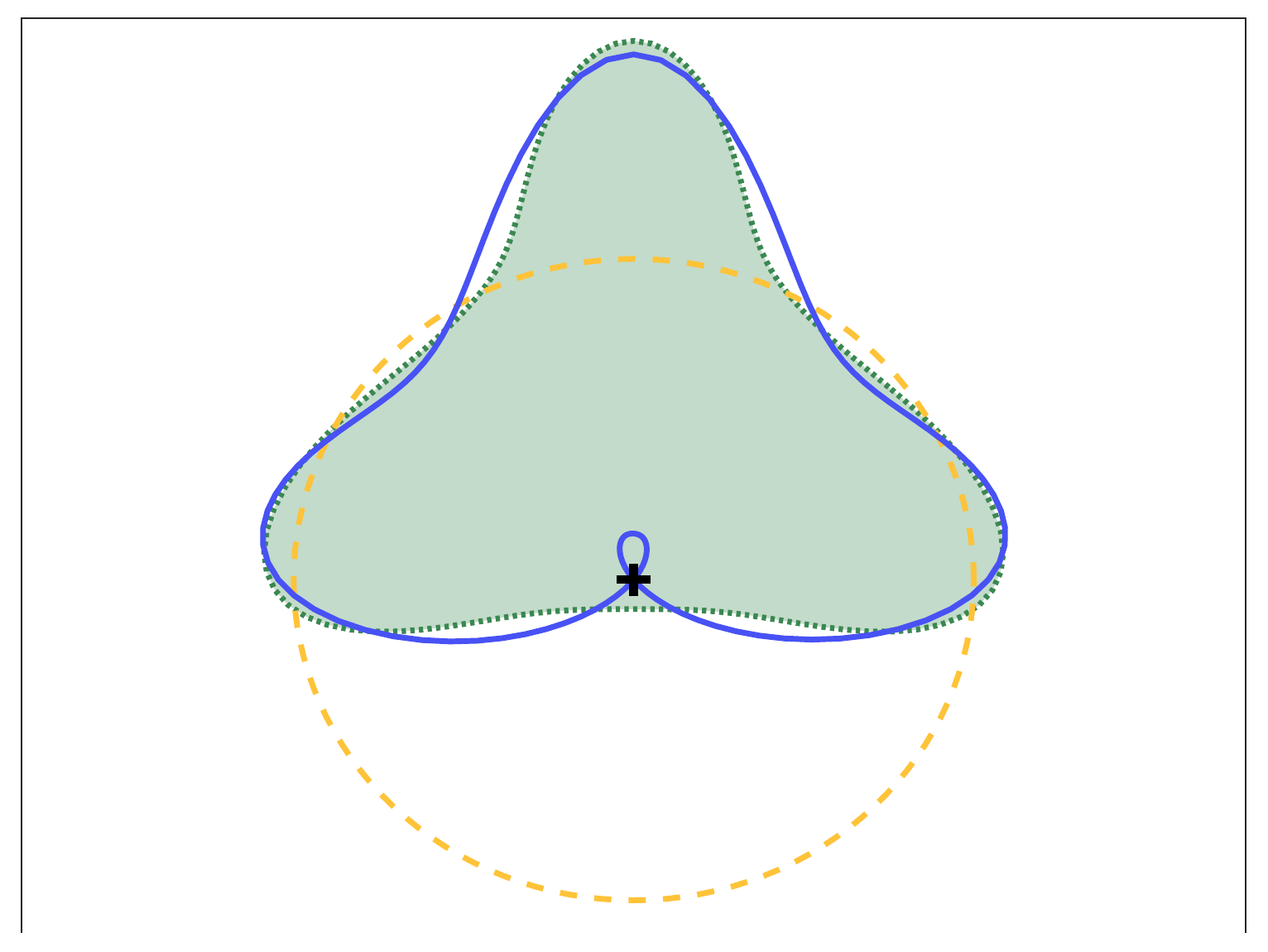}};
\begin{scope}[x={(image.south east)},y={(image.north west)}]
        \draw[black,thick] (0.32,0.2) rectangle (0.68,0.5);
    \end{scope}
\end{tikzpicture}
\subcaption{ }
\end{subfigure}
\begin{subfigure}[t]{.31\textwidth}
\begin{tikzpicture}
\node[anchor=south west,inner sep=0] (image) at (0,0) {\includegraphics[width=\textwidth,clip,trim={10 5 10 5}]{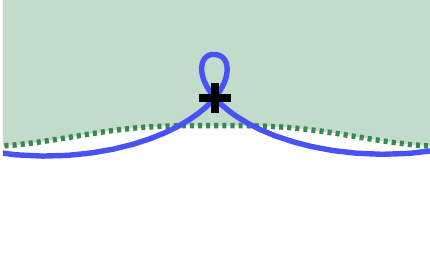}};
\begin{scope}[x={(image.south east)},y={(image.north west)}]
        \draw[black, thick] (0.,0.) rectangle (1.,1.);
    \end{scope}
\end{tikzpicture}
\subcaption{ }
\end{subfigure}
\begin{subfigure}[t]{.31\textwidth}
\includegraphics[width=\textwidth,clip,trim={20 5 20 10}]{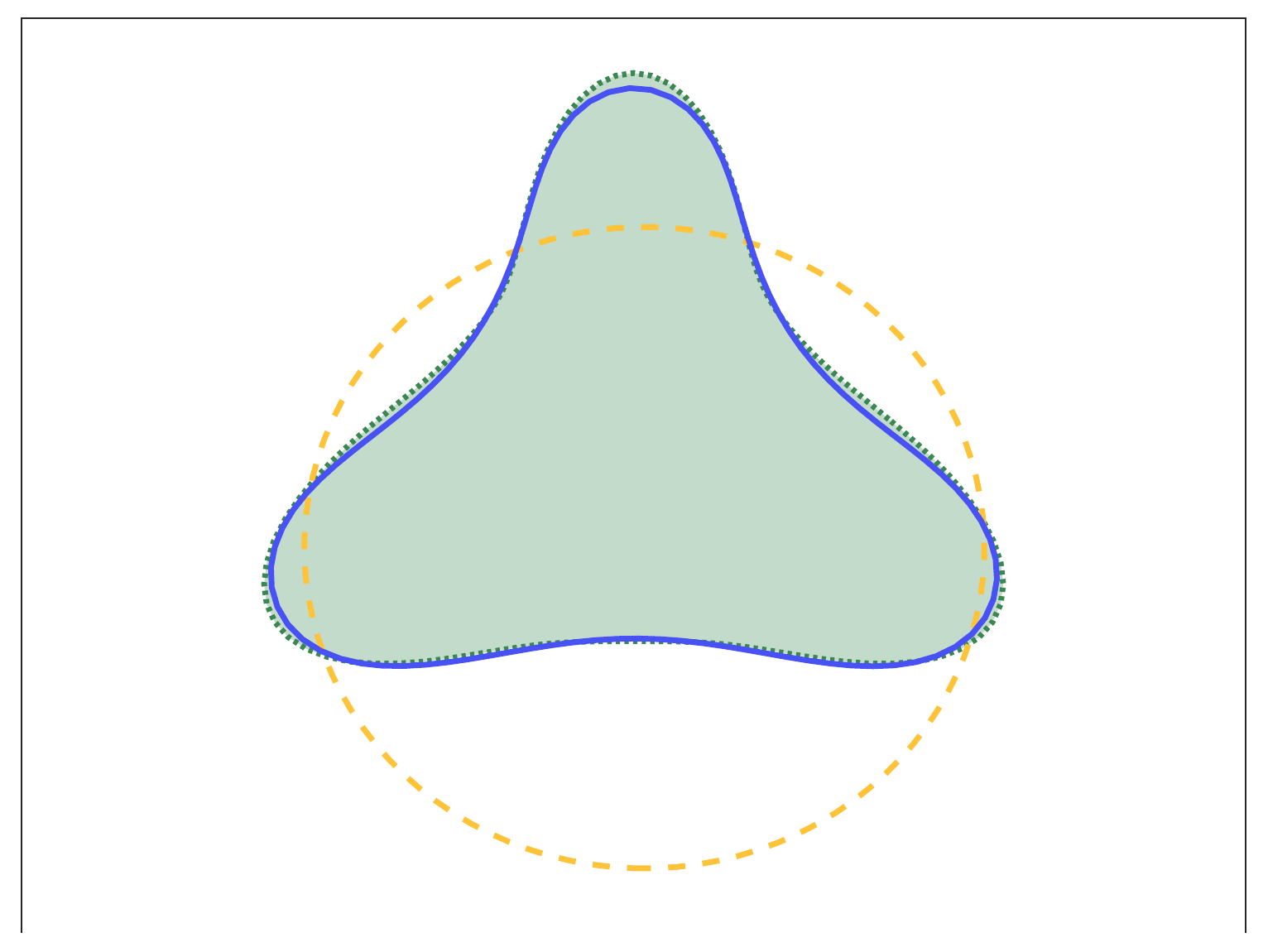}
\subcaption{ }
\end{subfigure}
\caption{\label{fig:rec_three_lobes} \footnotesize
Comparison of our method to previous radial function parameterizations 
for a star-shaped domain. 
Parameters and line styles are chosen as in Figure \ref{fig:rec_letter_s}. 
(a),(b),(d) show reconstructions using a radial function parameterization
with different choices of the center point indicated by a black cross; 
(c), (e) are magnifications of (b) and (d) around the center point, respectively; 
(f) shows the reconstruction by our bending energy approach.
}
\end{figure}

Figure \ref{fig:rec_three_lobes} (a) already illustrates the obvious limitation 
of the commonly used radial function parameterizations to star-shaped domains. 
In Figure \ref{fig:rec_three_lobes} we demonstrate a further disadvantage 
of such parameterizations, which is the dependence on the 
choice of the center point. 
We can observe unwanted deformations in the reconstruction or even a failure if the center point is chosen too  close to the 
boundary of the exact domain. This is expected since the penalty term corresponding 
to the exact solution explodes as the origin tends to the boundary. 
In contrast, in the proposed approach based on the bending energy  
the position of the obstacle with respect to the origin has 
no influence on the global minimum of the Tikhonov functional (although 
local minimization methods will get stuck in local minima if the initial guess 
is too far away from the true obstacle). 

We summarize that the proposed approach for solving inverse obstacle problems on a 
shape manifold with bending energy penalization may yield considerably better 
reconstructions than radial function parameterizations even for star-shaped obstacles
and allows the reconstruction of considerably more complicated curves. 


\section{Reconstructions with initial guesses provided by sampling methods} \label{sec:sampling}

\paragraph*{The factorization method} 
Let us briefly recall the factorization method as an example of a sampling method 
and typical numerical implementations of this method. 
Suppose that $\mathbb{M}=\Sb\times \Sb$ and denote the integral operator 
with kernel $\uinf$ by $U_\infty\in L(L^2(\Sb))$, i.e.\ 
\[
(\Uinf g)(\widehat{x})\ceq\int_{\Sb} \uinf(\widehat{x},d) \,g(d) 
\,\mathrm{d}s(d),\qquad
 \widehat{x}\in \Sb.
\]
Moreover, let $r_z(\widehat{x})\ceq\exp(-\ii \,k \, \widehat{x}^\top z)$ denote the far 
field pattern of a point source at $z\in\Rset^2$. 
The main result justifying the factorization method 
(see \cite[Theorem 3.8]{kirsch:98}) is that 
\[
r_z\in \mathrm{ran} (\Uinf^*\Uinf)^{1/4}\quad \Leftrightarrow\quad z\in \varOmega_{\mathrm{int}}.
\]
In practice, given only a discrete and noisy version of $\Uinf$, one constructs 
an approximation $A$ to the operator $(\Uinf^*\Uinf)^{-1/4}$ (e.g.\ by a trunctated 
eigenvalue decomposition) and uses sublevel sets of the function 
\begin{equation}\label{eq:defi_chi}
\chi(z)\ceq\|A \, r_z\|^2
\end{equation}
as approximations of $\varOmega_{\mathrm{int}}$ since for continuous noiseless data 
$\chi(z)<\infty$ (with an appropriate definition of $\chi$) if and only if 
$z\in \varOmega_{\mathrm{int}}$. 
There are several variants concerning the choice of $A$ which follow the 
same pattern (see \cite{AL:15}). 

In order to find a parameterization of some level line of a function $\chi\in C^1(\Rset^2)$ 
in the form \eqref{eq:CurveRepresentation} we introduce the forward operator 
$F_{1/\chi} \colon \Mcurves \to L^2([0,2\pi])$ defined by 
\[
\paren{F_{1/\chi}(\curve)}(t) \ceq \frac{1}{\chi(\gamma_\curve(t))},\qquad t\in [0,2\pi].
\]
\begin{figure}[ht]
\centering
\begin{minipage}{.7\textwidth}
\includegraphics[width=\textwidth,clip,trim={40 0 20 0}]{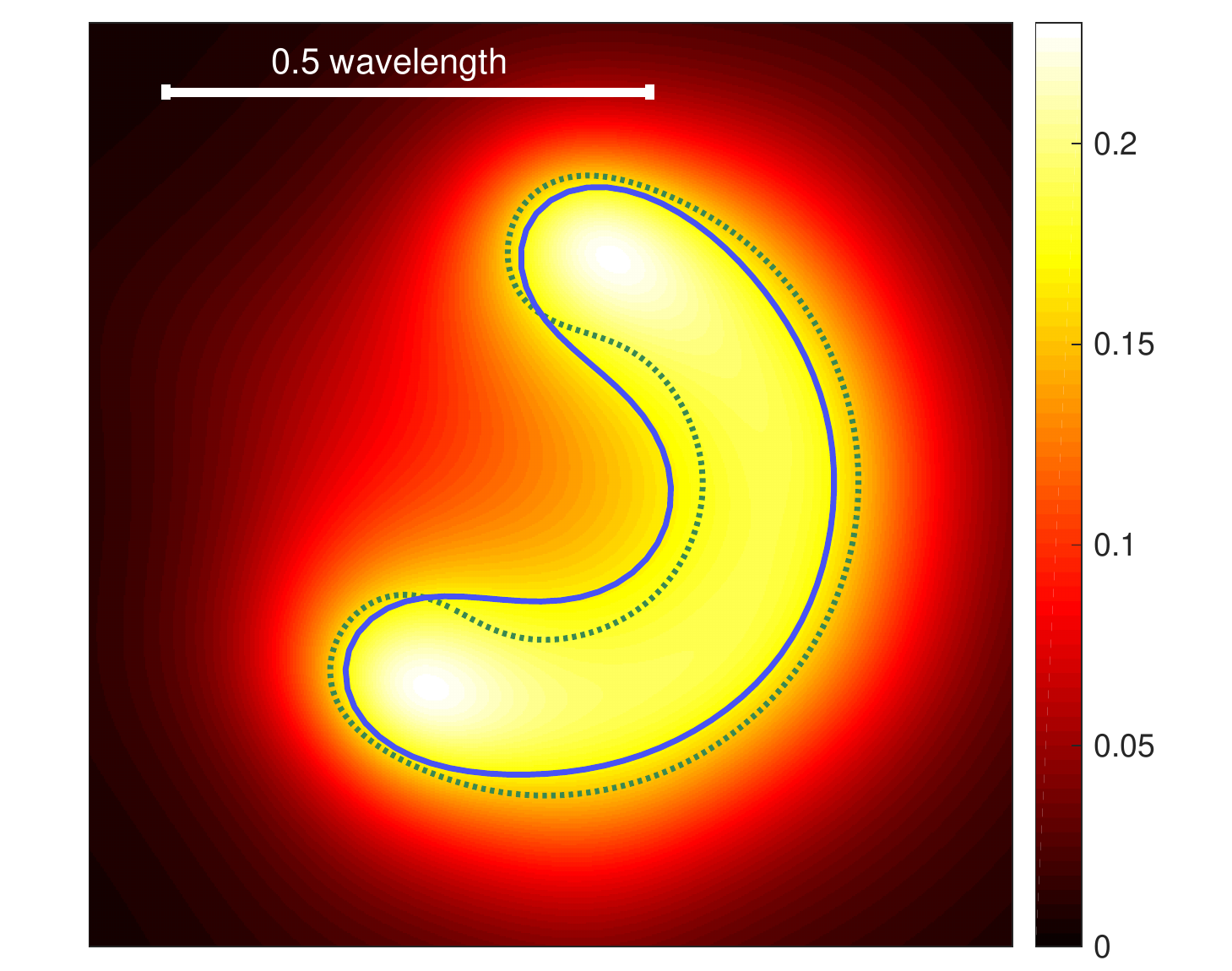}
\end{minipage}
\caption{\footnotesize 
Obstacle reconstruction  by parameterization of a level line for the factorization method. We use 20 incident 
waves and $1\%$ Gaussian white noise. 
The values of the function $1/\chi$ with $\chi$ given by \eqref{eq:defi_chi} 
are indicated by colors and the true obstacle by a dotted green line.
The solid blue line shows a parameterized level line of $1/\chi$ approximating 
the true obstacle.}
\label{fig:factorLevel_method_reconstruction}
\end{figure}

Then the problem to find a parameterization of the $\beta$-level-line of $\chi$ 
can be formulated as an operator equation $F_{1/\chi}(\curve)=\beta  \, \mathbf{1}$ 
where $\mathbf{1}\in L^2([0,2\pi])$ is the constant $1$ function. 
This problem may again be solved by Tikhonov regularization. 
The use of $F_{1/\chi}$ yields a more global convergence behavior than the use 
of $F_{\chi}$. 
Notice that $F_{1/\chi}$ is Fr\'echet differentiable with 
$(DF_{1/\chi}(\curve) \, h)(t)= -\chi(\curve(t))^{-2} \, \langle \grad \chi(\gamma_\curve(t)) , \gamma_h(t)\rangle$. 
Notice that for $\chi$ given by \eqref{eq:defi_chi}, we have 
$\partial_{z_j}\chi(z) =2 \Re \langle A \, r_z, A\, \partial_{z_j}r_z\rangle$. 

The reconstruction of a level line of $\chi$ (or equivalently $1/\chi$) is illustrated 
in Figure \ref{fig:factorLevel_method_reconstruction} using data corrupted by 
$1\%$ Gaussian white noise.

\paragraph*{Numerical results}
We now use the parameterization of the level line curve illustrated in 
Figure~\ref{fig:factorLevel_method_reconstruction} as an initial guess $\curve_*$ 
in Tikhonov regularization. The result is shown in 
Figure~\ref{fig:sausage_rec_comparison} (b). In most parts the reconstruction is 
hard to distinguish from the true curve by eye, and it is much better than 
a reconstruction using a circle as initial guess as shown in 
Figure~\ref{fig:sausage_rec_comparison} (a). 
Only in some interior parts of the ``horseshoe'' the reconstruction in (b) 
seems to take a ``short cut''. 
A reason may be that the initial guess curve $\gamma_{\curve_*}$ is ``too short'' in the 
interior part, and consequently geodesic distances of points on $\gamma_{\curve_*}$ 
\emph{relative} to its length $L_*$ do not match the geodesic distances of 
their best approximations on $\gamma_{\xdag}$ relative to $L^{\dagger}$. 
Therefore, the bending energy $\energy(\xdag,\curve_*)$ is quite large whereas due to the 
``short cut'' in the Tikhonov estimator $\curve_\alpha$ the bending energy
$\energy(\curve_\alpha,\curve_*)$ is much smaller. 
However, as illustrated in Figure \ref{fig:sausage_rec_comparison} (c), 
for smaller wavelengths the difference of the corresponding data fidelity terms 
becomes large enough to compensate for this effect.

\begin{figure}[ht]
\centering
\begin{subfigure}{.31\textwidth}
\includegraphics[width=\textwidth,clip,trim={70 5 70 5}]{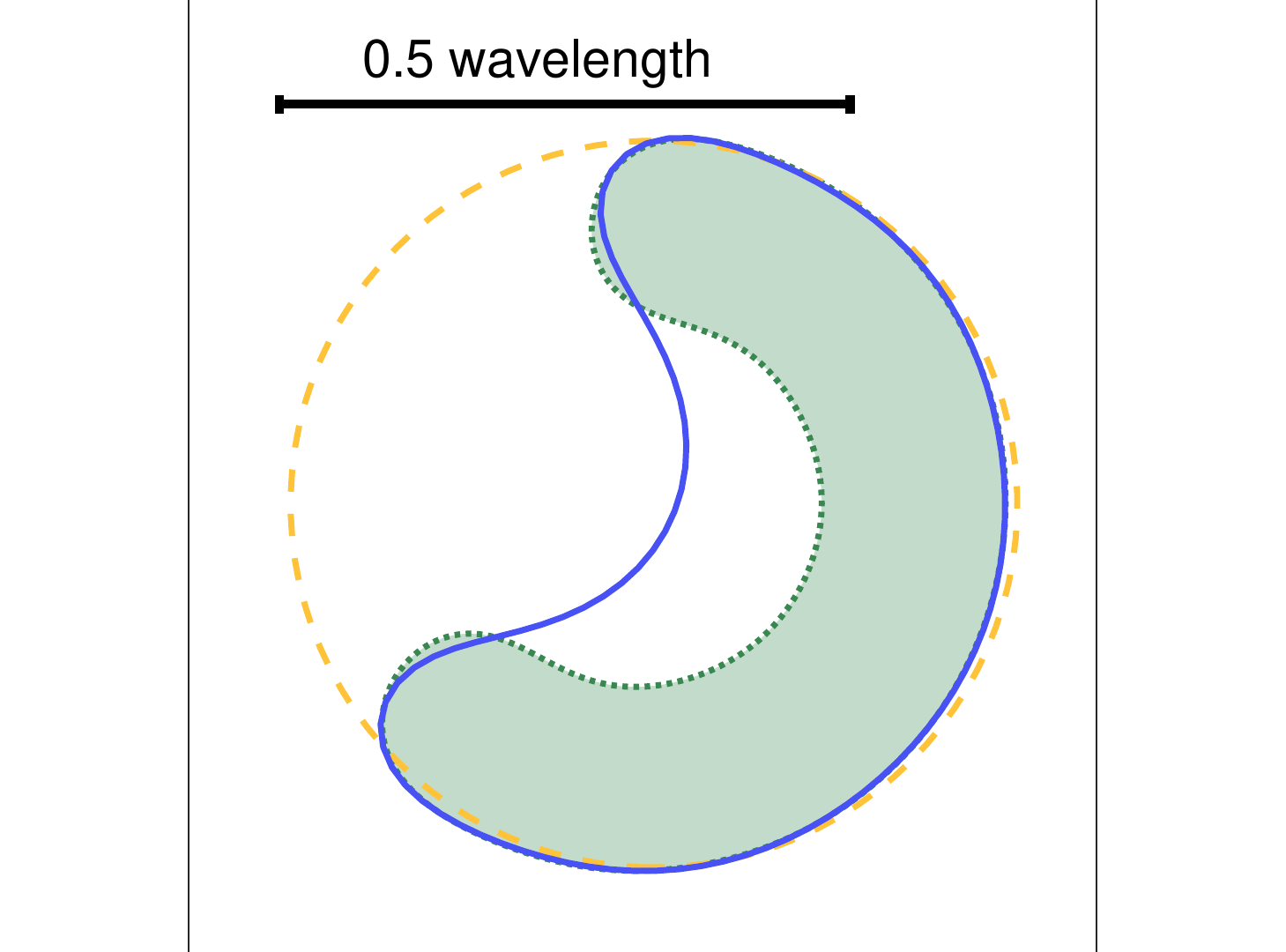}
\subcaption{ }
\end{subfigure}
\begin{subfigure}{.31\textwidth}
\includegraphics[width=\textwidth,clip,trim={70 5 70 5}]{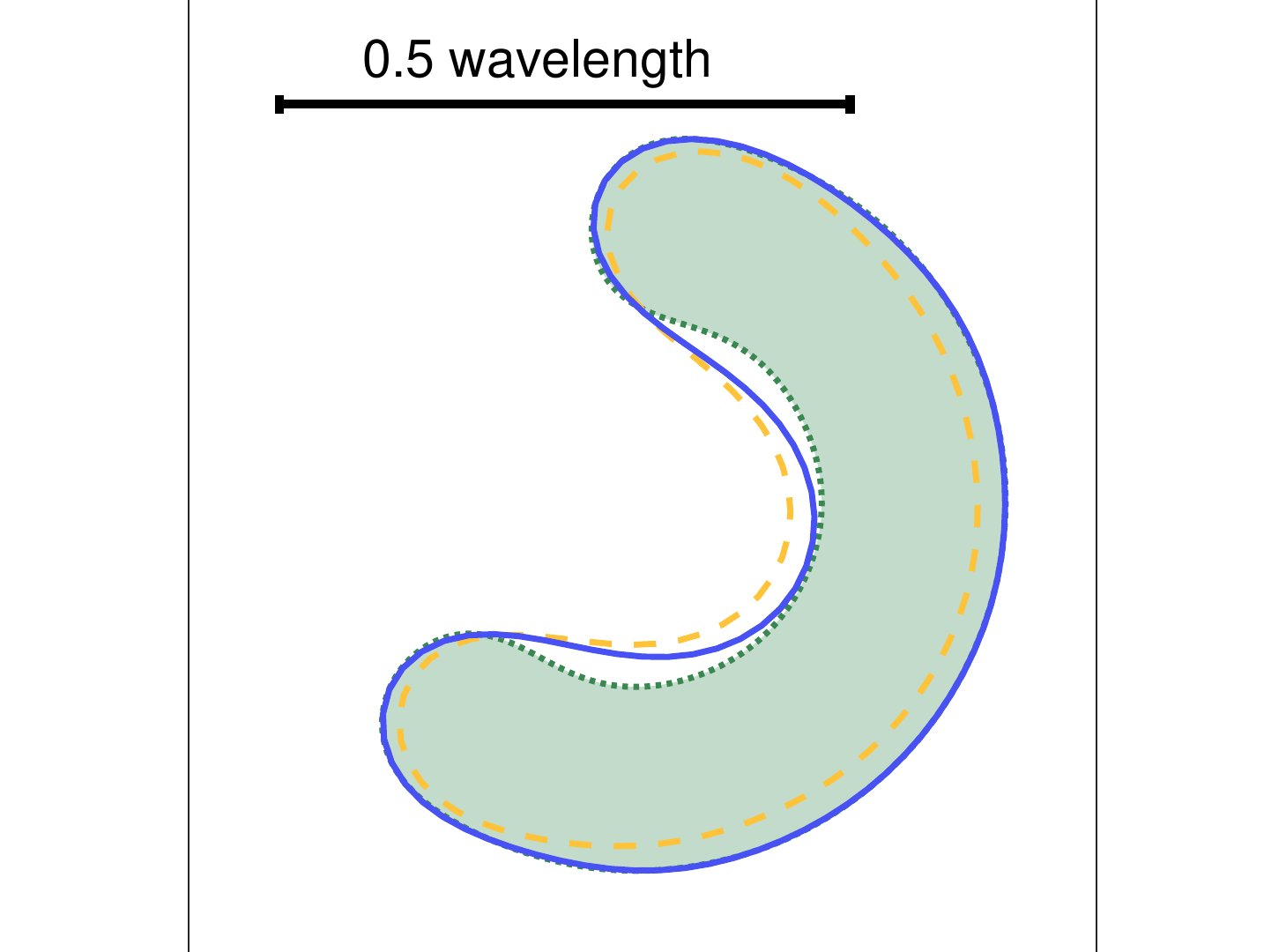}
\subcaption{ }
\end{subfigure}
\begin{subfigure}{.31\textwidth}
\includegraphics[width=\textwidth,clip,trim={70 5 70 5}]{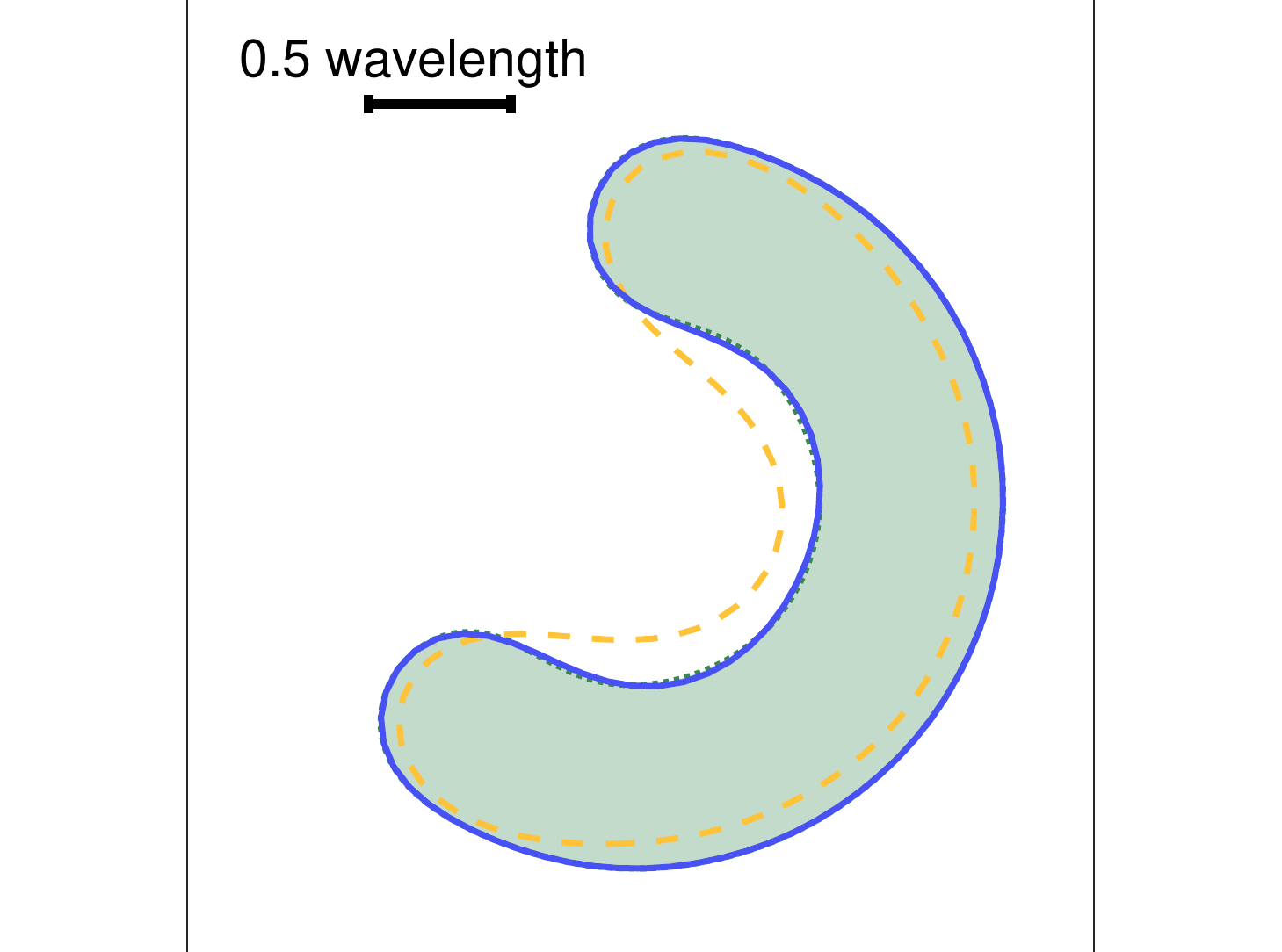}
\subcaption{ }
\end{subfigure}
\caption{\footnotesize Reconstructions by our method for different initial 
guesses and different wave numbers. We use far field data 
with $1\%$ Gaussian white noise and otherwise the same parameters and line styles 
as in Figure \ref{fig:rec_letter_s}. 
In panel (a) the initial guess (dashed yellow) 
is chosen as a circle, in panels (b) and (c) the initial guess is taken from the factorization method as 
illustrated in Figure \ref{fig:factorLevel_method_reconstruction}.  
The reconstruction in (c) uses far field data for a smaller wavelength.}
\label{fig:sausage_rec_comparison}
\end{figure}

\section*{Acknowledgement}
This paper is dedicated to the memory of Armin Lechleiter who has 
contributed substantially to the field of inverse scattering theory.  
We miss him as a friend and colleague. 

Financial support by the Deutsche Forschungsgemeinschaft through the RTG 2088 
is gratefully acknowledged. 

\bibliography{references}
\bibliographystyle{abbrv}
\end{document}